\documentclass[10pt, letterpaper]{article}
\usepackage{graphicx} 
\usepackage{subcaption}
\usepackage{placeins}
\usepackage{amsmath}
\usepackage{color}
\usepackage{amsthm}
\usepackage{amssymb}
\DeclareMathOperator*{\argmin}{arg\,min}
\usepackage[letterpaper, total={7in, 9in}]{geometry}
\usepackage{amsfonts}
\usepackage{stmaryrd}
\usepackage{authblk}
\usepackage{mathrsfs}

\theoremstyle{plain}
\newtheorem{theorem}{Theorem}[section]
\newtheorem{lemma}[theorem]{Lemma}
\newtheorem{proposition}[theorem]{Proposition}

\newtheorem{remark}[theorem]{Remark}

\title{Stable time-stepping via residual minimization: finite-element and neural-network approximations for transient parabolic problems \thanks{J.M.M. was supported by the Consolidated Research Group MATHMODE (IT1866-26) at UPV/EHU, funded by the Department of Education of the Basque Government, by Research Project PID2023-146668OA-I00, and by Grant RYC2023-045172-I funded by MICIU/AEI/10.13039/501100011033. S.R. was partially supported by ANID through FONDECYT Project 1240643 and by the National Center for Artificial Intelligence (CENIA FB210017), Basal ANID\@. Parts of this work were carried out during a research visit by the second author to the Basque Center for Applied Mathematics (BCAM, Spain). S.R. also thanks Prof. Carsten Carstensen for the initial discussion that led to this manuscript. } }

\author[1,2]{Judit Mu\~noz-Matute}
\author[3]{Sergio Rojas}

\affil[1]{\normalsize Department of Mathematics, The University of the Basque Country (EHU), 48940, Leioa, Spain
\protect\\ \texttt{judit.munoz@ehu.eus}}

\affil[2]{\normalsize IKERBASQUE, Basque Foundation for Science, 48009, Bilbao, Spain}

\affil[3]{\normalsize School of Mathematics, Monash University, 9 Rainforest Walk, 3800 Melbourne VIC, Australia
\protect\\ \texttt{sergio.rojas@monash.edu}}

\date{\today}

\begin{document}
\maketitle

\begin{abstract}
We propose a time-stepping minimum-residual (MinRes) framework for transient coercive variational problems, applicable to finite-element and neural-network trial approximations. For the Backward Euler scheme, we measure the residual at each time level in the dual norm induced by the steady test-space norm, scaled by the time step. This choice yields stability estimates controlling a discrete
parabolic energy quantity, up to the time-discretization defect. We develop conforming and broken-test formulations and specialize the analysis to diffusion--advection--reaction problems. For finite-dimensional test spaces, we derive fully discrete reliability estimates by supplementing the computable discrete residual with a complementary contribution that accounts for residual components the test space does not resolve. We then extend the framework to neural-network trial classes and obtain computable residual decompositions for conforming and broken polynomial test spaces. Numerical experiments confirm the expected finite-element convergence rates, examine the effect of test-space refinement for a global space--time neural approximation, and demonstrate residual-driven spatial refinement for a transient problem with a moving localized feature.
\end{abstract}

\section{Introduction}

Minimum-residual (MinRes) discretizations provide a systematic approach to stably approximating variational problems. For a well-posed linear problem satisfying the Banach--Ne\v{c}as--Babu\v{s}ka conditions \cite{babuvska1971error}, the approximation error is equivalent, up to stability constants, to the dual norm of the residual:
$$
\|u-u_\star\|_U
\eqsim
\|\ell(\cdot)-a(u_\star,\cdot)\|_{V'}
=
\|\varepsilon\|_V,
$$
where $\varepsilon\in V$ is the Riesz representative of the residual. A MinRes method therefore seeks an approximation $u_\star$ by minimizing a computable realization of this residual norm. This viewpoint underlies discontinuous Petrov--Galerkin (DPG) methods \cite{demkowicz2011analysis,demkowicz2012class}, where the choice of the test norm is part of the formulation and determines the residual norm being minimized; see the recent review \cite{demkowiczGopalakrishnan2025}. It also underlies residual-minimization approaches based on neural-network trial classes, such as robust variational physics-informed neural networks (RVPINNs)~\cite{rojas2024rvpinn}.

Extending this viewpoint to transient problems adds another difficulty. To make this explicit, consider the general variational parabolic problem:
$$
(\partial_t u(t),v)_0+a(u(t),v)
=
(f(t),v)_0,
\qquad
\forall v\in V,
$$
where $V$ is a Hilbert space continuously embedded in $L^2(\Omega)$ and $a(\cdot,\cdot)$ is a continuous and coercive bilinear form. For a uniform partition of the time interval $0=t_0<t_1<\cdots<t_{N+1}=T$, with time step size $\tau=t_{n+1}-t_n$, the Backward Euler discretization gives
$$
(u^{n+1},v)_0+\tau a(u^{n+1},v)
=
(u^n,v)_0+\tau(f^{n+1},v)_0,
\qquad
\forall v\in V,
$$
where $(\cdot,\cdot)_0$ denotes the usual inner product in $L^2(\Omega)$. For the heat equation, where $a(v,w)=(v,w)_V$, the time-step operator naturally suggests the inner product $(v,w)_0+\tau(v,w)_V$. From a minimum-residual perspective, however, adopting this inner product as the test-space inner product makes the norm itself depend on $\tau$ and couples the $L^2(\Omega)$ and elliptic scales. Consequently, the residual representative and the associated approximation properties also inherit this time-step dependence, and obtaining estimates in the usual discrete-in-time parabolic energy quantities becomes less direct than in the classical Galerkin analysis.

Time-stepping minimum-residual and DPG formulations have been studied previously in this setting \cite{roberts2021time}. F\"uhrer, Heuer, and Sen Gupta~\cite{fuhrerHeuerSenGupta2017} analyzed a Backward Euler DPG scheme based on an ultraweak formulation of the heat equation, using time-step-dependent weighted norms, and obtained stability and quasi-optimal error estimates. Related time-stepping least-squares formulations for parabolic first-order systems were studied by F\"uhrer and Karkulik~\cite{fuhrerKarkulik2019}, where an elliptic projection is introduced to recover optimal estimates. For primal DPG methods, F\"uhrer, Heuer, and Karkulik~\cite{fuhrerHeuerKarkulik2021} showed that optimal Backward Euler estimates can be obtained by using projections associated with the spatial part of the operator rather than with the full time-step problem. In the particular case of the heat equation, they also showed that the primal DPG field variable coincides with the standard Galerkin approximation. These works demonstrate that the interaction between the time-step-dependent problem and the spatial minimum-residual structure requires suitably weighted norms or projection arguments. Space--time least-squares \cite{fuhrerKarkulik2021spacetime,gantner2021further} and space--time DPG formulations \cite{chakraborty2025space,diening2022space} provide an alternative in which space and time are discretized simultaneously.

Rather than using the inner product induced by the complete Backward Euler time-step operator, we retain the steady spatial test-space inner product and scale it only by the time step, namely $\tau(v,w)_V$. Thus, the spatial norm used to represent the residual does not change with $\tau$; the time-step size appears only as an overall scaling. This simple choice yields an error equation that can be analyzed with the same type of energy argument used for the classical Backward Euler Galerkin method \cite{di2011mathematical}: testing with the error, using the polarization identity for the $L^2(\Omega)$ term, and exploiting the coercivity of the spatial bilinear form. In this sense, the resulting MinRes analysis retains the structure of the standard Galerkin stability argument while allowing independent choices of trial and test spaces.

More precisely, for an arbitrary trial sequence $\{u_\star^n\}_{n=0}^{N+1}$, let $e_\star^n:=u^n-u_\star^n$. The associated residual representative satisfies
$$
(\varepsilon_\star^{n+1},v)_V
=
\left(
\frac{e_\star^{n+1}-e_\star^n}{\tau},v
\right)_0
+
a(e_\star^{n+1},v)
-
(\vartheta^{n+1},v)_0,
\qquad
\forall v\in V,
$$
where $\vartheta^{n+1}$ denotes the Backward Euler consistency defect. The central stability estimate has the form
$$
\begin{aligned}
\|e_\star^{N+1}\|_0
&+
\left(
\sum_{n=0}^{N}
\|e_\star^{n+1}-e_\star^n\|_0^2
+
\tau
\sum_{n=0}^{N}
\|e_\star^{n+1}\|_V^2
\right)^{1/2}
\lesssim
\left(
\tau
\sum_{n=0}^{N}
\|\varepsilon_\star^{n+1}\|_V^2
\right)^{1/2}
+
\|e_\star^0\|_0
+
\mathcal O(\tau).
\end{aligned}
$$
The argument does not require the trial sequence to be generated by a finite-dimensional linear space or by a particular minimum-residual algorithm. This representation-independent step separates the time-stepping stability analysis from the subsequent discretization of the residual norm.

For neural-network trial classes, this separation has an additional computational consequence. We consider a single network $u_\theta:\Omega\times[0,T]\to\mathbb R$, which, for sufficiently smooth activation functions, provides a smooth approximation in both space and time. However, the residual is evaluated only at the discrete time levels. A direct space--time variational formulation would generally require numerical integration over $\Omega\times(0,T)$, whereas the present time-stepping construction replaces this by a sequence of spatial quadratures over $\Omega$. Thus, the neural trial function remains a global space--time approximation while the numerical integration used to construct the loss remains $d$-dimensional rather than $(d+1)$-dimensional. This distinction is relevant for variational neural methods, for which quadrature accuracy and cost are an important part of the residual evaluation \cite{kharazmi2021hpvpinn, berrone2022quadratures}. The resulting formulation therefore combines a smooth space--time neural representation with time-discrete residual evaluation, while the estimates below provide error control at the discrete time levels.

A closely related time-discrete VPINN strategy has recently been used for nonlinear heat-conduction problems, combining classical time discretization with minimization of a residual dual norm at each time step~\cite{bastidas2026timediscrete}. Here we instead derive representation-independent stability and efficiency estimates and quantify the effect of replacing the continuous residual norm by a finite-dimensional one. The same time-discrete construction could also be considered for strong-residual neural formulations such as PINNs \cite{raissi2019physics,shin2023error}. 

When the test space is discretized, an additional issue is whether the finite-dimensional test space resolves the continuous residual representative sufficiently well. This distinction is particularly important for neural-network trial classes. In RVPINNs~\cite{rojas2024rvpinn}, the variational residual is measured through a discrete dual norm computed from a finite-dimensional test space. The a posteriori framework in~\cite{fuhrer2025posteriori} makes explicit that a small discrete residual alone need not control the error if the test space does not detect the full residual; reliable control is recovered by supplementing the detected residual with a complementary residual contribution. A related approach adapts the test space itself. In \cite{udomworarat2026neural}, adaptive test-space enrichment is used to control the discrepancy between the discrete and continuous Riesz representatives, providing a mechanism for controlling the spatial discretization error in the discrete dual norm. The complementary residual contributions in this work play a related role: they quantify the component of the continuous residual that the chosen finite-dimensional test space misses.

We establish stability and efficiency estimates in abstract conforming and broken-test settings, and derive fully discrete residual comparisons after discretization of the test space. A diffusion--advection--reaction operator provides a concrete realization with computable local residual contributions. We also extend the framework to global space--time neural-network trial classes.

The remainder of the paper is organized as follows. Sections~\ref{sec:prelim}--\ref{sec:nn} develop the analytical framework, Section~\ref{sec:numerics} presents the numerical experiments, and the final section contains concluding remarks.

\section{Preliminaries}\label{sec:prelim}

\subsection{Abstract problem and notation}\label{subsec:model}

Let $\Omega\subset\mathbb{R}^d$, $d\in\{1,2,3\}$, be a bounded Lipschitz domain, and let $T>0$ be the final time. We denote by $(\cdot,\cdot)_0$ and $\|\cdot\|_0$ the inner product and norm in $L^2(\Omega)$, respectively. For brevity, we write $C^\ell(X):=C^\ell([0,T];X)$, where $X$ is a Hilbert space. We set $V:=H_0^1(\Omega)$, with $(v,w)_V:=(\nabla v,\nabla w)_0$ and $\|v\|_V:=\|\nabla v\|_0$. By the Poincar\'e inequality, there exists $C_P>0$ such that
\begin{equation}\label{eq:poincare}
\|v\|_0\le C_P\|v\|_V,
\qquad
\forall v\in V.
\end{equation}
We consider the abstract parabolic problem: find $u\in\mathcal X:=C^0(V)\cap C^1(L^2(\Omega))$ such that $u(0)=u_0$ and
\begin{equation}\label{eq:abstract-parabolic}
(\partial_tu(t),v)_0+a(u(t),v)
=
(f(t),v)_0,
\qquad
\forall v\in V,
\quad
t\in[0,T],
\end{equation}
where $f\in C^0(L^2(\Omega))$, $u_0\in V$, and $a:V\times V\to\mathbb R$ is a time-independent bilinear form satisfying
\begin{align}
|a(w,v)|
&\le M\|w\|_V\|v\|_V,
&&\forall w,v\in V,
\label{eq:a-continuity}\\
a(v,v)
&\ge \alpha\|v\|_V^2,
&&\forall v\in V,
\label{eq:a-coercive}
\end{align}
for constants $M,\alpha>0$. 

A principal realization of \eqref{eq:abstract-parabolic}, used below for computable residual decompositions and numerical experiments, is the diffusion--advection--reaction problem
\begin{equation}\label{eq:dar-strong}
\begin{aligned}
\partial_tu
-\epsilon\Delta u
+\boldsymbol\beta\cdot\nabla u
+\gamma u
&=f
&&\text{in }\Omega\times(0,T),\\
u&=0
&&\text{on }\partial\Omega\times(0,T),\\
u(0)&=u_0
&&\text{in }\Omega,
\end{aligned}
\end{equation}
where $\epsilon>0$, $\boldsymbol\beta\in W^{1,\infty}(\Omega)^d$, and $\gamma\in L^\infty(\Omega)$ are time-independent. The associated conforming bilinear form is
\begin{equation}\label{eq:dar-bilinear}
a_{\rm dar}(w,v)
:=
\epsilon(\nabla w,\nabla v)_0
+
(\boldsymbol\beta\cdot\nabla w,v)_0
+
(\gamma w,v)_0.
\end{equation}
If
\begin{equation}\label{eq:dar-coercivity-assumption}
\gamma-\frac12\nabla\cdot\boldsymbol\beta\ge0
\qquad\text{a.e.\ in }\Omega,
\end{equation}
then integration by parts gives
\begin{equation}\label{eq:dar-coercivity}
a_{\rm dar}(v,v)
=
\epsilon\|\nabla v\|_0^2
+
\left(\gamma-\frac12\nabla\cdot\boldsymbol\beta,v^2\right)_0
\ge
\epsilon\|v\|_V^2.
\end{equation}
Thus $a_{\rm dar}$ satisfies \eqref{eq:a-continuity}--\eqref{eq:a-coercive} with constants depending on the fixed coefficients. The heat equation is recovered by taking $\epsilon=1$, $\boldsymbol\beta=\boldsymbol0$, and $\gamma=0$.

\subsection{Conforming and broken variational settings}
\label{subsec:variational-formulations}

\paragraph{Conforming setting.}
The conforming setting is precisely \eqref{eq:abstract-parabolic}; the analysis below uses only the continuity and coercivity properties \eqref{eq:a-continuity}--\eqref{eq:a-coercive}.

\paragraph{Broken-test setting.}
Let $\mathcal T_h$ be a conforming, shape-regular partition of $\Omega$ into Lipschitz elements. We write $\Gamma_h:=\partial\mathcal T_h:=\bigcup_{K\in\mathcal T_h}\partial K$ for the mesh skeleton and define
\begin{equation}\label{eq:broken-space}
H^1(\mathcal T_h)
:=
\left\{
v\in L^2(\Omega):
v|_K\in H^1(K)
\quad\forall K\in\mathcal T_h
\right\}.
\end{equation}
For $v\in H^1(\mathcal T_h)$, let $\nabla_hv$ denote the elementwise gradient. We set $V_{\rm br}:=H^1(\mathcal T_h)$ and equip it with the inner product $(v,z)_{V_{\rm br}}:=(v,z)_0+(\nabla_hv,\nabla_hz)_0$, with corresponding norm $\|v\|_{V_{\rm br}}^2=\|v\|_0^2+\|\nabla_hv\|_0^2$. Since $V\subset V_{\rm br}$, the Poincar\'e inequality gives
\begin{equation}\label{eq:br-emb}
\|v\|_{V_{\rm br}}
\le
C_{\rm emb}\|v\|_V,
\qquad
C_{\rm emb}:=\sqrt{1+C_P^2},
\qquad
\forall v\in V.
\end{equation}
Moreover,
\begin{equation}\label{eq:br-l2-control}
\|v\|_0\le\|v\|_{V_{\rm br}},
\qquad
\forall v\in V_{\rm br}.
\end{equation}
Thus, $C_{0,\rm br}=1$ in the $L^2$-control used below. Let
$$
H(\operatorname{div};\Omega)
:=
\left\{
\boldsymbol q\in[L^2(\Omega)]^d:
\operatorname{div}\boldsymbol q\in L^2(\Omega)
\right\},
$$
 and define the broken normal-trace operator $\operatorname{tr}^{\operatorname{div}}_{\mathcal T_h}: H(\operatorname{div};\Omega)\to H^1(\mathcal T_h)'$ by
\begin{equation}\label{eq:broken-trace-operator}
\left\langle
\operatorname{tr}^{\operatorname{div}}_{\mathcal T_h}\boldsymbol q,v
\right\rangle_{\Gamma_h}
:=
\sum_{K\in\mathcal T_h}
\left[
(\operatorname{div}\boldsymbol q,v)_K
+
(\boldsymbol q,\nabla v)_K
\right],
\qquad
v\in H^1(\mathcal T_h).
\end{equation}
We define
\begin{equation}\label{eq:broken-flux-space}
\widehat\Lambda
:=
H^{-1/2}(\Gamma_h)
:=
\operatorname{ran}
\left(\operatorname{tr}^{\operatorname{div}}_{\mathcal T_h}\right)
\end{equation}
and equip it with the quotient norm
\begin{equation}\label{eq:broken-flux-norm}
\|\widehat\mu\|_{\widehat\Lambda}
:=
\inf_{\substack{
\boldsymbol q\in H(\operatorname{div};\Omega)\\
\operatorname{tr}^{\operatorname{div}}_{\mathcal T_h}\boldsymbol q=\widehat\mu
}}
\|\boldsymbol q\|_{H(\operatorname{div};\Omega)}.
\end{equation}
Let
\begin{equation}\label{eq:broken-trial-space}
U:=V\times\widehat\Lambda
\end{equation}
and define
\begin{equation}\label{eq:broken-trial-norm}
\|\boldsymbol w\|_U^2
:=
\|w\|_V^2+\|\widehat\mu\|_{\widehat\Lambda}^2,
\qquad
\boldsymbol w=(w,\widehat\mu)\in U.
\end{equation}
The corresponding space--time space is
\begin{equation}\label{eq:space-time-broken}
\mathcal X_{\rm br}
:=
\mathcal X\times C^0(\widehat\Lambda).
\end{equation}
We consider a bilinear form $a_{\rm br}:U\times V_{\rm br}\to\mathbb R$ satisfying the boundedness condition
\begin{equation}\label{eq:broken-continuity}
|a_{\rm br}(\boldsymbol w,v)|
\le
M_{\rm br}\|\boldsymbol w\|_U\|v\|_{V_{\rm br}},
\qquad
\forall \boldsymbol w\in U,
\quad
\forall v\in V_{\rm br},
\end{equation}
for a mesh-independent constant $M_{\rm br}>0$. We also assume that the associated stationary broken problem is well-posed and satisfies the mesh-independent inf--sup condition
\begin{equation}\label{eq:broken-below}
\beta_{\rm br}\|\boldsymbol w\|_U
\le
\sup_{0\ne v\in V_{\rm br}}
\frac{|a_{\rm br}(\boldsymbol w,v)|}
{\|v\|_{V_{\rm br}}},
\qquad
\forall \boldsymbol w\in U,
\end{equation}
for some $\beta_{\rm br}>0$. Finally, the time-stepping analysis requires the field-coercivity property
\begin{equation}\label{eq:broken-field-coercivity}
a_{\rm br}(\boldsymbol w,w)
=
a(w,w)
\ge
\alpha\|w\|_V^2,
\qquad
\forall\boldsymbol w=(w,\widehat\mu)\in U.
\end{equation}
We assume that the solution $u$ of \eqref{eq:abstract-parabolic} admits a skeleton variable $\widehat\lambda\in C^0(\widehat\Lambda)$ such that $\boldsymbol u:=(u,\widehat\lambda)\in\mathcal X_{\rm br}$ satisfies
\begin{equation}\label{eq:abstract-broken}
(\partial_tu(t),v)_0
+
a_{\rm br}(\boldsymbol u(t),v)
=
(f(t),v)_0,
\qquad
\forall v\in V_{\rm br},
\quad
t\in[0,T].
\end{equation}
For the diffusion--advection--reaction problem \eqref{eq:dar-strong}, the corresponding broken bilinear form is
\begin{equation}\label{eq:dar-broken-bilinear}
\begin{aligned}
a_{{\rm dar},{\rm br}}
\bigl((w,\widehat\mu),v\bigr)
:={}&
\epsilon(\nabla w,\nabla_hv)_0
+
(\boldsymbol\beta\cdot\nabla w,v)_0
+
(\gamma w,v)_0
-
\langle\widehat\mu,v\rangle_{\Gamma_h}.
\end{aligned}
\end{equation}
Its boundedness with respect to the $U$- and $V_{\rm br}$-norms follows from the coefficient bounds and continuity of the normal-trace pairing. For this realization, we assume that the associated stationary broken formulation satisfies the mesh-independent inf--sup condition \eqref{eq:broken-below}. For the diffusion-only primal DPG formulation, the corresponding stability property is standard; see, e.g., \cite[Section~3]{DemkowiczGopalakrishnan2013}.

The field-coercivity property follows directly. If $\widehat\mu\in\widehat\Lambda$ and $w\in V$, then $\widehat\mu=\operatorname{tr}^{\rm div}_{\mathcal T_h}\boldsymbol q$ for some $\boldsymbol q\in H(\operatorname{div};\Omega)$. Since $w$ has zero trace on $\partial\Omega$, Green's formula gives
\begin{equation}\label{eq:trace-cancellation}
\langle\widehat\mu,w\rangle_{\Gamma_h}
=
(\operatorname{div}\boldsymbol q,w)_0
+
(\boldsymbol q,\nabla w)_0
=
0.
\end{equation}
Consequently,
$$
a_{{\rm dar},{\rm br}}
\bigl((w,\widehat\mu),w\bigr)
=
a_{\rm dar}(w,w),
$$
and therefore \eqref{eq:broken-field-coercivity} follows from \eqref{eq:dar-coercivity}. For $u\in\mathcal X$ satisfying \eqref{eq:dar-strong}, the compatible skeleton variable is the physical diffusive normal trace
\begin{equation}\label{eq:exact-skeleton}
\widehat\lambda(t)
:=
\operatorname{tr}^{\rm div}_{\mathcal T_h}
\bigl(\epsilon\nabla u(t)\bigr).
\end{equation}
Indeed, \eqref{eq:dar-strong} gives
$$
\operatorname{div}(\epsilon\nabla u(t))
=
\partial_tu(t)
+
\boldsymbol\beta\cdot\nabla u(t)
+
\gamma u(t)
-
f(t)
\in L^2(\Omega).
$$
Since the coefficients are time independent and bounded, $u\in C^0(V)$ and $\partial_tu,f\in C^0(L^2(\Omega))$ imply
$$
\epsilon\nabla u
\in
C^0\bigl([0,T];H(\operatorname{div};\Omega)\bigr).
$$
Continuity of the normal-trace operator therefore gives $\widehat\lambda\in C^0(\widehat\Lambda)$. With $\boldsymbol u=(u,\widehat\lambda)$, elementwise integration by parts shows that \eqref{eq:abstract-broken} is the broken-test formulation associated with \eqref{eq:dar-strong}.

\section{Semidiscrete Riesz representatives}\label{sec:ie-ideal}

In this section, we develop the semidiscrete residual framework for the abstract conforming and broken-test settings of Section~\ref{sec:prelim}. This argument applies to finite-element and neural-network approximations and does not assume that the trial sequence is generated by a finite-dimensional linear space or a particular minimum-residual algorithm. Given an arbitrary sequence of approximations at the discrete time levels, we associate with it a \emph{semidiscrete Riesz representative} of the residual induced by the time-marching scheme. For simplicity, we develop the analysis for Backward Euler. The same residual-scaling strategy adapts to other implicit time-stepping schemes, with the corresponding stability argument and consistency terms depending on the scheme.

\noindent
Let $0=t_0<t_1<\cdots<t_{N+1}=T$ be a uniform partition of $[0,T]$, with $\tau:=t_{n+1}-t_n$. We define the Backward Euler difference operator by
\begin{equation}\label{eq:be-diffop}
\delta_t w^{n+1}:=\frac{w^{n+1}-w^n}{\tau}.
\end{equation}
For the consistency analysis, we additionally assume $\partial_{tt}u\in L^2(0,T;L^2(\Omega))$. Under this regularity assumption, we can write
\begin{equation}\label{eq:be-defect}
\partial_t u^{n+1}
=
\delta_t u^{n+1}-\vartheta^{n+1},
\qquad
\vartheta^{n+1}
:=
\frac{1}{\tau}
\int_{t_n}^{t_{n+1}}
(t_n-t)\,\partial_{tt}u(t)\,dt.
\end{equation}
Here, $\vartheta^{n+1}$ is the Backward Euler consistency defect at time $t_{n+1}$. By the Cauchy--Schwarz inequality,
$$
\tau\sum_{n=0}^{N}\|\vartheta^{n+1}\|_0^2
\le
\frac{\tau^2}{3}
\|\partial_{tt}u\|_{L^2(0,T;L^2(\Omega))}^2.
$$
Hence, the consistency contribution in the squared stability estimates is of order $\tau^2$.

\subsection{Conforming formulation}\label{sec:standard_wf}
Using \eqref{eq:be-defect}, we infer that the exact solution of \eqref{eq:abstract-parabolic} satisfies the time-discrete relation
\begin{equation}\label{eq:exact-be-standard}
(u^{n+1},v)_0+\tau a(u^{n+1},v)
=
(u^n,v)_0+\tau(f^{n+1},v)_0+\tau(\vartheta^{n+1},v)_0,
\qquad \forall v\in V.
\end{equation}
Let $\{u_\star^n\}_{n=0}^{N+1}\subset V$ be any sequence of trial approximations.
We define the errors
$$
e_\star^n:=u^n-u_\star^n,
\qquad n=0,\dots,N+1,
$$
and the associated semidiscrete Riesz representative $\varepsilon_\star^{n+1}\in V$ by
\begin{equation}\label{eq:riesz-standard-def}
\tau(\varepsilon_\star^{n+1},v)_V
:=
(u_\star^n-u_\star^{n+1},v)_0
+\tau(f^{n+1},v)_0
-\tau a(u_\star^{n+1},v),
\qquad \forall v\in V.
\end{equation}
Subtracting the residual identity \eqref{eq:riesz-standard-def} from \eqref{eq:exact-be-standard} gives
\begin{equation}\label{eq:standard-error-eq}
(e_\star^{n+1},v)_0+\tau a(e_\star^{n+1},v)
=
(e_\star^n,v)_0+\tau(\varepsilon_\star^{n+1},v)_V+\tau(\vartheta^{n+1},v)_0,
\qquad \forall v\in V.
\end{equation}
\begin{theorem}[Semidiscrete stability: conforming formulation]
\label{thm:standard-stability}
For the semidiscrete Riesz representatives defined by \eqref{eq:riesz-standard-def}, the following estimate holds:
\begin{equation}\label{eq:standard-stability}
\|e_\star^{N+1}\|_0^2
+
\sum_{n=0}^{N}\|e_\star^{n+1}-e_\star^n\|_0^2
+
\tau\alpha\sum_{n=0}^{N}\|e_\star^{n+1}\|_V^2
\le
\|e_\star^0\|_0^2
+
\frac{2\tau}{\alpha}\sum_{n=0}^{N}\|\varepsilon_\star^{n+1}\|_V^2
+
\frac{2\tau C_P^2}{\alpha}\sum_{n=0}^{N}\|\vartheta^{n+1}\|_0^2.
\end{equation}
\end{theorem}

\begin{proof}
We test \eqref{eq:standard-error-eq} with $v=e_\star^{n+1}\in V$. This gives
$$
\|e_\star^{n+1}\|_0^2+\tau a(e_\star^{n+1},e_\star^{n+1})
=
(e_\star^n,e_\star^{n+1})_0
+\tau(\varepsilon_\star^{n+1},e_\star^{n+1})_V
+\tau(\vartheta^{n+1},e_\star^{n+1})_0.
$$
Using the identity
$$
(e_\star^n,e_\star^{n+1})_0
=
\frac12\Big(
\|e_\star^n\|_0^2+\|e_\star^{n+1}\|_0^2-\|e_\star^{n+1}-e_\star^n\|_0^2
\Big),
$$
we obtain
$$
\|e_\star^{n+1}\|_0^2
+\|e_\star^{n+1}-e_\star^n\|_0^2
+2\tau a(e_\star^{n+1},e_\star^{n+1})
=
\|e_\star^n\|_0^2
+2\tau(\varepsilon_\star^{n+1},e_\star^{n+1})_V
+2\tau(\vartheta^{n+1},e_\star^{n+1})_0.
$$
By the Cauchy--Schwarz inequality and Young's inequality,
$$
2\tau(\varepsilon_\star^{n+1},e_\star^{n+1})_V
\le
2\tau\|\varepsilon_\star^{n+1}\|_V\|e_\star^{n+1}\|_V
\le
\frac{2\tau}{\alpha}\|\varepsilon_\star^{n+1}\|_V^2
+\frac{\tau\alpha}{2}\|e_\star^{n+1}\|_V^2.
$$
Moreover, by Poincar\'e's inequality,
$$
2\tau(\vartheta^{n+1},e_\star^{n+1})_0
\le
2\tau\|\vartheta^{n+1}\|_0\|e_\star^{n+1}\|_0
\le
2\tau C_P\|\vartheta^{n+1}\|_0\|e_\star^{n+1}\|_V
\le
\frac{2\tau C_P^2}{\alpha}\|\vartheta^{n+1}\|_0^2
+\frac{\tau\alpha}{2}\|e_\star^{n+1}\|_V^2.
$$
Using the coercivity \eqref{eq:a-coercive}, we infer
$$
\|e_\star^{n+1}\|_0^2
+\|e_\star^{n+1}-e_\star^n\|_0^2
+\tau\alpha\|e_\star^{n+1}\|_V^2
\le
\|e_\star^n\|_0^2
+\frac{2\tau}{\alpha}\|\varepsilon_\star^{n+1}\|_V^2
+\frac{2\tau C_P^2}{\alpha}\|\vartheta^{n+1}\|_0^2.
$$
Finally, summing over $n=0,\dots,N$ yields \eqref{eq:standard-stability}.
\end{proof}

\begin{theorem}[Semidiscrete efficiency: conforming formulation]
\label{thm:standard-efficiency}
The semidiscrete Riesz representatives satisfy
\begin{equation}\label{eq:standard-efficiency}
\tau\sum_{n=0}^{N}\|\varepsilon_\star^{n+1}\|_V^2
\le
\frac{4C_P^2}{\tau}\sum_{n=0}^{N}\|e_\star^{n+1}-e_\star^n\|_0^2
+
4\tau M^2\sum_{n=0}^{N}\|e_\star^{n+1}\|_V^2
+
4\tau C_P^2\sum_{n=0}^{N}\|\vartheta^{n+1}\|_0^2.
\end{equation}
\end{theorem}

\begin{proof}
Testing \eqref{eq:standard-error-eq} with $v=\varepsilon_\star^{n+1}\in V$, we obtain
$$
\tau\|\varepsilon_\star^{n+1}\|_V^2
=
(e_\star^{n+1}-e_\star^n,\varepsilon_\star^{n+1})_0
+
\tau a(e_\star^{n+1},\varepsilon_\star^{n+1})
-
\tau(\vartheta^{n+1},\varepsilon_\star^{n+1})_0.
$$
Using Cauchy--Schwarz, Poincar\'e's inequality, and continuity of $a(\cdot,\cdot)$, we obtain
\begin{align*}
\tau\|\varepsilon_\star^{n+1}\|_V^2
&\le
\|e_\star^{n+1}-e_\star^n\|_0\,\|\varepsilon_\star^{n+1}\|_0
+
\tau M\|e_\star^{n+1}\|_V\|\varepsilon_\star^{n+1}\|_V
+
\tau\|\vartheta^{n+1}\|_0\,\|\varepsilon_\star^{n+1}\|_0
\\
&\le
C_P\|e_\star^{n+1}-e_\star^n\|_0\,\|\varepsilon_\star^{n+1}\|_V
+
\tau M\|e_\star^{n+1}\|_V\|\varepsilon_\star^{n+1}\|_V
+
\tau C_P\|\vartheta^{n+1}\|_0\,\|\varepsilon_\star^{n+1}\|_V.
\end{align*}
Applying Young's inequality to each product,
\begin{align*}
C_P\|e_\star^{n+1}-e_\star^n\|_0\,\|\varepsilon_\star^{n+1}\|_V
&\le
\frac{\tau}{4}\|\varepsilon_\star^{n+1}\|_V^2
+
\frac{C_P^2}{\tau}\|e_\star^{n+1}-e_\star^n\|_0^2,
\\
\tau M\|e_\star^{n+1}\|_V\|\varepsilon_\star^{n+1}\|_V
&\le
\frac{\tau}{4}\|\varepsilon_\star^{n+1}\|_V^2
+
\tau M^2\|e_\star^{n+1}\|_V^2,
\\
\tau C_P\|\vartheta^{n+1}\|_0\,\|\varepsilon_\star^{n+1}\|_V
&\le
\frac{\tau}{4}\|\varepsilon_\star^{n+1}\|_V^2
+
\tau C_P^2\|\vartheta^{n+1}\|_0^2.
\end{align*}
Hence,
$$
\tau\|\varepsilon_\star^{n+1}\|_V^2
\le
\frac{3\tau}{4}\|\varepsilon_\star^{n+1}\|_V^2
+
\frac{C_P^2}{\tau}\|e_\star^{n+1}-e_\star^n\|_0^2
+
\tau M^2\|e_\star^{n+1}\|_V^2
+
\tau C_P^2\|\vartheta^{n+1}\|_0^2 ,
$$
and therefore
$$
\tau\|\varepsilon_\star^{n+1}\|_V^2
\le
\frac{4C_P^2}{\tau}\|e_\star^{n+1}-e_\star^n\|_0^2
+
4\tau M^2\|e_\star^{n+1}\|_V^2
+
4\tau C_P^2\|\vartheta^{n+1}\|_0^2.
$$
Summing over $n=0,\dots,N$ concludes the proof.
\end{proof}

\subsection{Broken-test formulation}
\label{subsec:semidiscrete-broken}

We now consider the broken-test formulation introduced in Section~\ref{subsec:variational-formulations}. Let $\boldsymbol u_\star^n =(u_\star^n,\widehat\lambda_\star^n)\in U$, $n=0,\ldots,N+1$, be any sequence of trial approximations. We define $\boldsymbol e_\star^n :=\boldsymbol u^n-\boldsymbol u_\star^n =(e_\star^n,\widehat e_\star^{\,n})$, where $e_\star^n:=u^n-u_\star^n$ and $\widehat e_\star^{\,n} :=\widehat\lambda^n-\widehat\lambda_\star^n$. We then define the broken semidiscrete Riesz representative $\varepsilon_{\mathrm{br},\star}^{n+1}\in V_{\mathrm{br}}$ by
\begin{equation}\label{eq:riesz-broken-def}
\tau(\varepsilon_{\mathrm{br},\star}^{n+1},v)_{V_{\mathrm{br}}}
:=
(u_\star^n-u_\star^{n+1},v)_0
+\tau(f^{n+1},v)_0
-\tau a_{\mathrm{br}}(\boldsymbol{u}_\star^{n+1},v),
\qquad \forall v\in V_{\mathrm{br}}.
\end{equation}
Using the exact time-discrete broken formulation,
$$
(u^{n+1},v)_0+\tau a_{\mathrm{br}}(\boldsymbol{u}^{n+1},v)
=
(u^n,v)_0+\tau(f^{n+1},v)_0+\tau(\vartheta^{n+1},v)_0,
\qquad \forall v\in V_{\mathrm{br}},
$$
we immediately obtain the corresponding error equation
\begin{equation}\label{eq:broken-error-eq}
(e_\star^{n+1},v)_0+\tau a_{\mathrm{br}}(\boldsymbol{e}_\star^{n+1},v)
=
(e_\star^n,v)_0+\tau(\varepsilon_{\mathrm{br},\star}^{n+1},v)_{V_{\mathrm{br}}}
+\tau(\vartheta^{n+1},v)_0,
\qquad \forall v\in V_{\mathrm{br}}.
\end{equation}
For the analysis below, we will use the continuous embedding and $L^2$-control \eqref{eq:br-emb}--\eqref{eq:br-l2-control}.
\begin{theorem}[Semidiscrete stability: broken-test formulation]
\label{thm:broken-stability}
Under the above assumptions,
\begin{equation}\label{eq:broken-stability}
\|e_\star^{N+1}\|_0^2
+
\sum_{n=0}^{N}\|e_\star^{n+1}-e_\star^n\|_0^2
+
\tau\alpha\sum_{n=0}^{N}\|e_\star^{n+1}\|_V^2
\le
\|e_\star^0\|_0^2
+
\frac{2\tau C_{\mathrm{emb}}^2}{\alpha}
\sum_{n=0}^{N}\|\varepsilon_{\mathrm{br},\star}^{n+1}\|_{V_{\mathrm{br}}}^2
+
\frac{2\tau C_P^2}{\alpha}
\sum_{n=0}^{N}\|\vartheta^{n+1}\|_0^2.
\end{equation}
\end{theorem}

\begin{proof}
The argument follows the proof of Theorem~\ref{thm:standard-stability}. Since $e_\star^{n+1}\in V\subset V_{\rm br}$, we may test
\eqref{eq:broken-error-eq} with $v=e_\star^{n+1}$. Using the identity
$$
2(e_\star^n,e_\star^{n+1})_0
=
\|e_\star^n\|_0^2
+
\|e_\star^{n+1}\|_0^2
-
\|e_\star^{n+1}-e_\star^n\|_0^2,
$$
together with \eqref{eq:broken-field-coercivity}, \eqref{eq:br-emb}, the Poincar\'e inequality, and Young's inequality, gives
$$
\begin{aligned}
\|e_\star^{n+1}\|_0^2
&+
\|e_\star^{n+1}-e_\star^n\|_0^2
+
\tau\alpha\|e_\star^{n+1}\|_V^2
\\
&\le
\|e_\star^n\|_0^2
+
\frac{2\tau C_{\rm emb}^2}{\alpha}
\|\varepsilon_{{\rm br},\star}^{n+1}\|_{V_{\rm br}}^2
+
\frac{2\tau C_P^2}{\alpha}
\|\vartheta^{n+1}\|_0^2.
\end{aligned}
$$
Summing over $n=0,\ldots,N$ proves the result.
\end{proof}

\begin{theorem}[Semidiscrete efficiency: broken-test formulation]
\label{thm:broken-efficiency}
Under the assumptions of Theorem~\ref{thm:broken-stability}, the following estimate holds:
\begin{equation}\label{eq:broken-efficiency}
\tau\sum_{n=0}^{N}\|\varepsilon_{\mathrm{br},\star}^{n+1}\|_{V_{\mathrm{br}}}^2
\le
\frac{4C_{0,\mathrm{br}}^2}{\tau}\sum_{n=0}^{N}\|e_\star^{n+1}-e_\star^n\|_0^2
+
4\tau M_{\mathrm{br}}^2\sum_{n=0}^{N}\|\boldsymbol{e}_\star^{n+1}\|_U^2
+
4\tau C_{0,\mathrm{br}}^2\sum_{n=0}^{N}\|\vartheta^{n+1}\|_0^2.
\end{equation}
\end{theorem}

\begin{proof}
Testing \eqref{eq:broken-error-eq} with $v=\varepsilon_{\rm br,\star}^{n+1}$ and using
\eqref{eq:br-l2-control} and the continuity of $a_{\rm br}$ gives
$$
\begin{aligned}
\tau
\|\varepsilon_{{\rm br},\star}^{n+1}\|_{V_{\rm br}}^2
\le{}&
C_{0,{\rm br}}
\|e_\star^{n+1}-e_\star^n\|_0
\|\varepsilon_{{\rm br},\star}^{n+1}\|_{V_{\rm br}}
\\
&+
\tau M_{\rm br}
\|\boldsymbol e_\star^{n+1}\|_U
\|\varepsilon_{{\rm br},\star}^{n+1}\|_{V_{\rm br}}
\\
&+
\tau C_{0,{\rm br}}
\|\vartheta^{n+1}\|_0
\|\varepsilon_{{\rm br},\star}^{n+1}\|_{V_{\rm br}}.
\end{aligned}
$$
Applying Young's inequality to the three terms and summing over $n=0,\ldots,N$ yields the result.
\end{proof}
\begin{remark}[Stability versus efficiency]
The efficiency estimates above do not provide a $\tau$-uniform reverse bound for the error quantities appearing in Theorems~\ref{thm:standard-stability} and~\ref{thm:broken-stability}. In particular, the estimates for the residual representatives contain the stronger discrete time-increment term
$$
\frac{1}{\tau}
\sum_{n=0}^{N}
\|e_\star^{n+1}-e_\star^n\|_0^2,
$$
whereas the corresponding term in the stability estimates is
$$
\sum_{n=0}^{N}
\|e_\star^{n+1}-e_\star^n\|_0^2.
$$
Accordingly, the stability and efficiency estimates should not be interpreted as a two-sided equivalence in the discrete energy quantity used in the stability theorems, uniformly with respect to $\tau$. For the broken formulation there is an additional asymmetry: Theorem~\ref{thm:broken-efficiency} involves the full trial error $\|\boldsymbol e_\star^{n+1}\|_U$, including the skeleton component, whereas Theorem~\ref{thm:broken-stability} controls the parabolic energy of the field component. Thus no two-sided equivalence for the full broken trial error is asserted either.
\end{remark}

\section{Fully discrete FE-based MinRes formulations}\label{sec:fem-fullydisc}

We now introduce fully discrete finite-dimensional realizations of the semidiscrete framework from Section~\ref{sec:ie-ideal}. At each time level, we define the new approximation by minimizing the corresponding discrete residual in a dual norm induced by a finite-dimensional test space. This applies to both the conforming and broken formulations, and yields mixed saddle-point systems involving a discrete residual representative.

\subsection{Conforming formulation}

For the conforming case, let $X_h\subset V$ and $Y_h\subset V$ be finite-dimensional trial and test spaces, respectively. We assume that $Y_h$ is equipped with an inner product $(\cdot,\cdot)_{Y_h}$, with induced norm $\|\cdot\|_{Y_h}$, uniformly equivalent to $\|\cdot\|_V$ on $Y_h$. We further assume that the discrete time-step operator is injective on $X_h$, namely,
\begin{equation}\label{eq:discrete-injectivity-standard}
(w_h,v_h)_0+\tau a(w_h,v_h)=0
\quad \forall v_h\in Y_h
\qquad\Longrightarrow\qquad
w_h=0.
\end{equation}
In particular, this implies $\dim(X_h)\leq\dim(Y_h)$. 

The fully discrete MinRes approximation at time $t_{n+1}$ is then defined by
\begin{equation}\label{eq:minres-standard}
u_h^{n+1}\in \argmin_{w_h\in X_h}\eta_h^{n+1}(w_h),
\end{equation}
where
\begin{equation}\label{eq:eta-standard-discrete}
\eta_h^{n+1}(w_h)
:=
\sup_{0\neq v_h\in Y_h}
\frac{
(u_h^n,v_h)_0+\tau(f^{n+1},v_h)_0-(w_h,v_h)_0-\tau a(w_h,v_h)
}{\sqrt{\tau}\|v_h\|_{Y_h}}.
\end{equation}
Equivalently, introducing the discrete residual representative $\varepsilon_h^{n+1}\in Y_h$ through
\begin{equation}\label{eq:riesz-standard-discrete}
\tau(\varepsilon_h^{n+1},v_h)_{Y_h}
=
(u_h^n,v_h)_0+\tau(f^{n+1},v_h)_0-(u_h^{n+1},v_h)_0-\tau a(u_h^{n+1},v_h),
\qquad \forall v_h\in Y_h,
\end{equation}
we have that
\begin{equation}\label{eq:eta-eps-standard}
\eta_h^{n+1}(u_h^{n+1})=\sqrt{\tau}\,\|\varepsilon_h^{n+1}\|_{Y_h}.
\end{equation}
%
The following standard mixed characterization of the minimum-residual method follows from the first-order optimality conditions; see, e.g., \cite[Theorems~3.4 and~6.3]{demkowiczGopalakrishnan2025}.
\begin{proposition}\label{prop:minres-standard-mixed}
The following statements are equivalent:
\begin{itemize}
\item[(i)] $u_h^{n+1}\in X_h$ is the unique solution of the minimization problem \eqref{eq:minres-standard}.
\item[(ii)] There exists $\varepsilon_h^{n+1}\in Y_h$ such that $(\varepsilon_h^{n+1},u_h^{n+1})\in Y_h\times X_h$ satisfies
\begin{equation}\label{eq:mixed-standard}
\begin{aligned}
\tau(\varepsilon_h^{n+1},v_h)_{Y_h}
+(u_h^{n+1},v_h)_0+\tau a(u_h^{n+1},v_h)
&=(u_h^n,v_h)_0+\tau(f^{n+1},v_h)_0,
&&\forall v_h\in Y_h,\\
(z_h,\varepsilon_h^{n+1})_0+\tau a(z_h,\varepsilon_h^{n+1})
&=0,
&&\forall z_h\in X_h.
\end{aligned}
\end{equation}
\end{itemize}
Moreover, $\varepsilon_h^{n+1}$ is uniquely determined by $u_h^{n+1}$, and \eqref{eq:eta-eps-standard} holds.
\end{proposition}

When $X_h=Y_h$, the standard Backward Euler Galerkin solution has vanishing discrete residual and hence coincides with the unique MinRes solution. For the general case, subtracting \eqref{eq:mixed-standard} from the exact time-discrete identity \eqref{eq:exact-be-standard} gives the fully discrete error relation
\begin{equation}\label{eq:error-standard-fully-discrete}
(e_h^{n+1},v_h)_0+\tau a(e_h^{n+1},v_h)
=
(e_h^n,v_h)_0+\tau(\varepsilon_h^{n+1},v_h)_{Y_h}+\tau(\vartheta^{n+1},v_h)_0,
\qquad \forall v_h\in Y_h,
\end{equation}
where
$
e_h^n:=u^n-u_h^n.
$
The fully discrete error relation \eqref{eq:error-standard-fully-discrete} only holds for test functions in $Y_h$. Therefore, in contrast with the semidiscrete analysis of Section~\ref{sec:ie-ideal}, we cannot directly test with the full error $e_h^{n+1}\in V$. To recover a computable estimate, we compare the discrete residual with the corresponding residual acting on the whole test space.
For $w_h\in X_h$, define the time-step residual $\mathcal R_h^{n+1}(w_h)\in V'$ by
\begin{equation}\label{eq:continuous-residual-standard}
\mathcal R_h^{n+1}(w_h)(v)
:=
(u_h^n-w_h,v)_0
+\tau(f^{n+1},v)_0
-\tau a(w_h,v),
\qquad \forall v\in V.
\end{equation}
Thus, the fully discrete approximation $u_h^{n+1}$ minimizes the restriction of $\mathcal R_h^{n+1}(w_h)$ to $Y_h$, measured in the dual norm induced by $\sqrt{\tau}\|\cdot\|_{Y_h}$.
Let $\varepsilon_{\star,h}^{n+1}\in V$ denote the continuous Riesz representative of the residual generated by the computed approximation, namely
\begin{equation}\label{eq:continuous-riesz-standard-fully}
\tau(\varepsilon_{\star,h}^{n+1},v)_V
=
\mathcal R_h^{n+1}(u_h^{n+1})(v),
\qquad \forall v\in V.
\end{equation}
Then
$$
\sqrt{\tau}\|\varepsilon_{\star,h}^{n+1}\|_V
=
\|\mathcal R_h^{n+1}(u_h^{n+1})\|_{(\tau V)'},
$$
where $(\tau V)'$ denotes the dual norm induced by the scaled norm
$\sqrt{\tau}\|\cdot\|_V$.
Let $\Pi_h:V\to Y_h$ be a bounded projection or quasi-interpolation operator satisfying
\begin{equation}\label{eq:Pi-standard-bound}
\|\Pi_h v\|_{Y_h}\le C_\Pi \|v\|_V,
\qquad \forall v\in V,
\end{equation}
we define the complementary residual contribution by
\begin{equation}\label{eq:rho-standard-def}
\rho_h^{n+1}
:=
\sup_{0\neq v\in V}
\frac{
\mathcal R_h^{n+1}(u_h^{n+1})(v-\Pi_h v)
}
{\sqrt{\tau}\|v\|_V}.
\end{equation}
The quantity $\rho_h^{n+1}$ measures the part of the residual that is invisible to the discrete test space. The next result connects the computable discrete residual with the continuous residual entering the semidiscrete stability estimate.
\begin{lemma}\label{lem:eta-rho-standard}
Assume \eqref{eq:Pi-standard-bound} and that the norms
$\|\cdot\|_{Y_h}$ and $\|\cdot\|_V$ are uniformly equivalent on $Y_h$. Then
\begin{equation}\label{eq:eta-rho-standard-bound}
\eta_h^{n+1}(u_h^{n+1})
\lesssim
\sqrt{\tau}\|\varepsilon_{\star,h}^{n+1}\|_V
\le
C_\Pi \eta_h^{n+1}(u_h^{n+1})+\rho_h^{n+1}.
\end{equation}
\end{lemma}
\begin{proof}
The first inequality follows from $Y_h\subset V$ and the uniform equivalence of the norms. For the second one, for arbitrary $v\in V$ we write
$$
\mathcal R_h^{n+1}(u_h^{n+1})(v)
=
\mathcal R_h^{n+1}(u_h^{n+1})(\Pi_h v)
+
\mathcal R_h^{n+1}(u_h^{n+1})(v-\Pi_h v).
$$
The first term is bounded by
$$
\eta_h^{n+1}(u_h^{n+1})\,\sqrt{\tau}\|\Pi_hv\|_{Y_h}
\le
C_\Pi \eta_h^{n+1}(u_h^{n+1})\,\sqrt{\tau}\|v\|_V,
$$
while the second term is bounded by the definition of $\rho_h^{n+1}$. Dividing by $\sqrt{\tau}\|v\|_V$ and taking the supremum over $v\in V\setminus\{0\}$ gives the result.
\end{proof}
We can now combine Lemma~\ref{lem:eta-rho-standard} with the semidiscrete stability estimate of Theorem~\ref{thm:standard-stability}, applied to the particular trial sequence $\{u_h^n\}_{n=0}^{N+1}$.
\begin{theorem}[Fully discrete reliability: conforming formulation]
\label{thm:fully-discrete-standard-reliability}
Let $\{u_h^n\}_{n=0}^{N+1}\subset X_h$ be generated by
\eqref{eq:minres-standard}. Then
\begin{align}
&\|e_h^{N+1}\|_0^2
+
\sum_{n=0}^{N}\|e_h^{n+1}-e_h^n\|_0^2
+
\tau\alpha\sum_{n=0}^{N}\|e_h^{n+1}\|_V^2
\nonumber\\
&\qquad\le
\|e_h^0\|_0^2
+
\frac{2}{\alpha}
\sum_{n=0}^{N}
\left(
C_\Pi\eta_h^{n+1}(u_h^{n+1})+\rho_h^{n+1}
\right)^2
+
\frac{2\tau C_P^2}{\alpha}
\sum_{n=0}^{N}\|\vartheta^{n+1}\|_0^2.
\label{eq:fully-discrete-standard-reliability}
\end{align}
\end{theorem}
\begin{proof}
Apply Theorem~\ref{thm:standard-stability} to the trial sequence
$u_\star^n=u_h^n$. The corresponding semidiscrete residual
representative is $\varepsilon_{\star,h}^{n+1}$ defined in
\eqref{eq:continuous-riesz-standard-fully}.
Lemma~\ref{lem:eta-rho-standard} then gives
\eqref{eq:fully-discrete-standard-reliability}.
\end{proof}
For concrete differential operators, the complementary residual \eqref{eq:rho-standard-def} can be localized by standard residual-estimator arguments. For the diffusion case, see, e.g., \cite[Section~4.2.1]{fuhrer2025posteriori}.
%
%

\subsection{Broken-test formulation}
\label{subsec:fully-discrete-broken}

We now consider the broken-test formulation introduced in Section~\ref{subsec:variational-formulations} and analyzed semidiscretely in Section~\ref{subsec:semidiscrete-broken}. Let $\boldsymbol X_h\subset U$ and $Y_h^{\rm br}\subset V_{\rm br}$ be finite-dimensional trial and test spaces, respectively. We assume that $Y_h^{\rm br}$ is equipped with an inner product $(\cdot,\cdot)_{Y_h^{\rm br}}$, with induced norm $\|\cdot\|_{Y_h^{\rm br}}$, uniformly equivalent to $\|\cdot\|_{V_{\rm br}}$ on $Y_h^{\rm br}$. We further assume that the discrete broken time-step operator is injective on $\boldsymbol X_h$, namely,
\begin{equation}\label{eq:discrete-injectivity-broken}
(w_h,v_h)_0
+\tau a_{\rm br}(\boldsymbol w_h,v_h)
=0
\quad \forall v_h\in Y_h^{\rm br}
\qquad\Longrightarrow\qquad
\boldsymbol w_h=0,
\end{equation}
for $\boldsymbol w_h=(w_h,\widehat\mu_h)\in\boldsymbol X_h$.
In particular,
$
\dim(\boldsymbol X_h)\leq\dim(Y_h^{\rm br}).
$
The fully discrete broken MinRes problem is: find $\boldsymbol{u}_h^{n+1}\in \boldsymbol X_h$ such that
\begin{equation}\label{eq:minres-broken}
\boldsymbol{u}_h^{n+1}\in \argmin_{\boldsymbol{w}_h\in \boldsymbol X_h}\eta_{{\rm br},h}^{n+1}(\boldsymbol{w}_h),
\end{equation}
where
\begin{equation}\label{eq:eta-broken-discrete}
\eta_{{\rm br},h}^{n+1}(\boldsymbol{w}_h)
:=
\sup_{0\neq v_h\in Y_h^{\rm br}}
\frac{
(u_h^n,v_h)_0+\tau(f^{n+1},v_h)_0-(w_h,v_h)_0-\tau a_{\rm br}(\boldsymbol{w}_h,v_h)
}
{\sqrt{\tau}\|v_h\|_{Y_h^{\rm br}}}.
\end{equation}
Introducing the discrete broken residual representative $\varepsilon_{{\rm br},h}^{n+1}\in Y_h^{\rm br}$ through
\begin{equation}\label{eq:riesz-broken-discrete}
\tau(\varepsilon_{{\rm br},h}^{n+1},v_h)_{Y_h^{\rm br}}
=
(u_h^n,v_h)_0+\tau(f^{n+1},v_h)_0-(u_h^{n+1},v_h)_0-\tau a_{\rm br}(\boldsymbol{u}_h^{n+1},v_h),
\qquad \forall v_h\in Y_h^{\rm br},
\end{equation}
we again obtain
\begin{equation}\label{eq:eta-eps-broken}
\eta_{{\rm br},h}^{n+1}(\boldsymbol{u}_h^{n+1})
=
\sqrt{\tau}\,\|\varepsilon_{{\rm br},h}^{n+1}\|_{Y_h^{\rm br}}.
\end{equation}
The next result follows from the same minimum-residual optimality argument as Proposition~\ref{prop:minres-standard-mixed}. 
\begin{proposition}\label{prop:minres-broken-mixed}
The following statements are equivalent:
\begin{itemize}
\item[(i)] $\boldsymbol{u}_h^{n+1}\in \boldsymbol X_h$ is the unique
solution of \eqref{eq:minres-broken}.
\item[(ii)] There exists $\varepsilon_{{\rm br},h}^{n+1}\in Y_h^{\rm br}$ such that
$(\varepsilon_{{\rm br},h}^{n+1},\boldsymbol{u}_h^{n+1})\in Y_h^{\rm br}\times \boldsymbol X_h$ satisfies
\begin{equation}\label{eq:mixed-broken}
\begin{aligned}
\tau(\varepsilon_{{\rm br},h}^{n+1},v_h)_{Y_h^{\rm br}}
+(u_h^{n+1},v_h)_0+\tau a_{\rm br}(\boldsymbol{u}_h^{n+1},v_h)
&=
(u_h^n,v_h)_0+\tau(f^{n+1},v_h)_0,
&&\forall v_h\in Y_h^{\rm br},
\\
(z_h,\varepsilon_{{\rm br},h}^{n+1})_0+\tau a_{\rm br}(\boldsymbol{z}_h,\varepsilon_{{\rm br},h}^{n+1})
&=0,
&&\forall \boldsymbol{z}_h=(z_h,\widehat\mu_h)\in \boldsymbol X_h.
\end{aligned}
\end{equation}
\end{itemize}
Moreover, \eqref{eq:eta-eps-broken} holds.
\end{proposition}
Subtracting \eqref{eq:mixed-broken} from the exact broken time-discrete formulation yields the fully discrete error equation
\begin{equation}\label{eq:error-broken-fully-discrete}
(e_h^{n+1},v_h)_0+\tau a_{\rm br}(\boldsymbol{e}_h^{n+1},v_h)
=
(e_h^n,v_h)_0
+\tau(\varepsilon_{{\rm br},h}^{n+1},v_h)_{Y_h^{\rm br}}
+\tau(\vartheta^{n+1},v_h)_0,
\qquad \forall v_h\in Y_h^{\rm br},
\end{equation}
where
$$
\boldsymbol{e}_h^n:=\boldsymbol{u}^n-\boldsymbol{u}_h^n.
$$
\begin{remark}[Recovery of the standard Galerkin field]
\label{rem:recovery-galerkin-field}
If the broken test space is equipped with the time-step-dependent inner product
$$
(v_h,w_h)_{Y_h^{\rm br},\tau}
:=
\frac{1}{\tau}(v_h,w_h)_0
+
(\nabla_hv_h,\nabla_hw_h)_0,
$$
and the usual compatibility assumptions between the field and test spaces hold, then for the heat equation the field component of the primal DPG approximation coincides with the standard Backward Euler Galerkin approximation; see \cite[Section~2.2.3]{fuhrerHeuerKarkulik2021}. This identity does not hold for a general choice of the discrete test inner product.
\end{remark}
We next introduce the continuous broken residual associated with the computed approximation. For $\boldsymbol{w}_h=(w_h,\widehat\mu_h)\in \boldsymbol X_h$, define $\mathcal R_{{\rm br},h}^{n+1}(\boldsymbol{w}_h)\in V_{\rm br}'$ by
\begin{equation}\label{eq:continuous-residual-broken}
\mathcal R_{{\rm br},h}^{n+1}(\boldsymbol{w}_h)(v)
:=
(u_h^n-w_h,v)_0
+\tau(f^{n+1},v)_0
-\tau a_{\rm br}(\boldsymbol{w}_h,v),
\qquad \forall v\in V_{\rm br}.
\end{equation}
Let $\varepsilon_{{\rm br},\star,h}^{n+1}\in V_{\rm br}$ be the continuous broken Riesz representative defined by
\begin{equation}\label{eq:continuous-riesz-broken-fully}
\tau(\varepsilon_{{\rm br},\star,h}^{n+1},v)_{V_{\rm br}}
=
\mathcal R_{{\rm br},h}^{n+1}(\boldsymbol{u}_h^{n+1})(v),
\qquad \forall v\in V_{\rm br}.
\end{equation}
Thus
$$
\sqrt{\tau}\|\varepsilon_{{\rm br},\star,h}^{n+1}\|_{V_{\rm br}}
=
\|\mathcal R_{{\rm br},h}^{n+1}(\boldsymbol{u}_h^{n+1})\|_{(\tau V_{\rm br})'}.
$$
To compare the discrete and continuous residuals, let
$$
I_h^{\rm br}:V_{\rm br}\to Y_h^{\rm br}
$$
be a bounded projection or quasi-interpolation operator satisfying
\begin{equation}\label{eq:broken-interpolant-stability}
\|I_h^{\rm br}v\|_{Y_h^{\rm br}}
\le
C_I\|v\|_{V_{\rm br}},
\qquad
\forall v\in V_{\rm br},
\end{equation}
where $C_I>0$ is independent of $h$ and $\tau$. We define the complementary broken residual by
\begin{equation}\label{eq:rho-broken-fe-def}
\rho_{{\rm br},h}^{n+1}
:=
\sup_{0\neq v\in V_{\rm br}}
\frac{
\mathcal R_{{\rm br},h}^{n+1}(\boldsymbol{u}_h^{n+1})
(v-I_h^{\rm br}v)
}
{\sqrt{\tau}\|v\|_{V_{\rm br}}}.
\end{equation}
\begin{lemma}\label{lem:eta-rho-broken-fe}
Assume \eqref{eq:broken-interpolant-stability} and that the norms $\|\cdot\|_{Y_h^{\rm br}}$ and $\|\cdot\|_{V_{\rm br}}$ are uniformly equivalent on $Y_h^{\rm br}$. Then
\begin{equation}\label{eq:eta-rho-broken-fe-bound}
\eta_{{\rm br},h}^{n+1}(\boldsymbol{u}_h^{n+1})
\lesssim
\sqrt{\tau}
\|\varepsilon_{{\rm br},\star,h}^{n+1}\|_{V_{\rm br}}
\le
C_I
\eta_{{\rm br},h}^{n+1}(\boldsymbol{u}_h^{n+1})
+
\rho_{{\rm br},h}^{n+1}.
\end{equation}
\end{lemma}

\begin{proof}
The proof is identical to that of Lemma~\ref{lem:eta-rho-standard}, with $V$, $Y_h$, and $\Pi_h$ replaced by $V_{\rm br}$, $Y_h^{\rm br}$, and $I_h^{\rm br}$, respectively.
\end{proof}
\begin{theorem}[Fully discrete reliability: broken-test formulation]
\label{thm:fully-discrete-broken-reliability}
Let $\{\boldsymbol{u}_h^n\}_{n=0}^{N+1}\subset \boldsymbol X_h$ be generated by
\eqref{eq:minres-broken}. Then
\begin{align}
&\|e_h^{N+1}\|_0^2
+
\sum_{n=0}^{N}\|e_h^{n+1}-e_h^n\|_0^2
+
\tau\alpha\sum_{n=0}^{N} \|e_h^{n+1}\|_V^2
\nonumber\\
&\qquad\le
\|e_h^0\|_0^2
+
\frac{2C_{\rm emb}^2}{\alpha}
\sum_{n=0}^{N}
\left(
C_I \eta_{{\rm br},h}^{n+1}(\boldsymbol{u}_h^{n+1})
+
\rho_{{\rm br},h}^{n+1}
\right)^2
+
\frac{2\tau C_P^2}{\alpha}
\sum_{n=0}^{N}\|\vartheta^{n+1}\|_0^2.
\label{eq:fully-discrete-broken-reliability}
\end{align}
\end{theorem}
\begin{proof}
Apply Theorem~\ref{thm:broken-stability} to $\boldsymbol u_\star^n=\boldsymbol u_h^n$. The corresponding semidiscrete broken residual representative is $\varepsilon_{{\rm br},\star,h}^{n+1}$ defined in \eqref{eq:continuous-riesz-broken-fully}. Lemma~\ref{lem:eta-rho-broken-fe} then gives \eqref{eq:fully-discrete-broken-reliability}.
\end{proof}
\begin{remark}[Localization of the complementary residual]
\label{rem:rho-broken-localization}
For broken polynomial test spaces, the Riesz problems defining $\eta_{{\rm br},h}^{n+1}$ are local, element by element. To obtain an explicit localization of the complementary residual, we specialize in this remark to the diffusion--advection--reaction realization
\eqref{eq:dar-strong}. The residual $\rho_{{\rm br},h}^{n+1}$ can then be localized provided that the operator $I_h^{\rm br}$ introduced above satisfies the usual local approximation properties. Assume in addition that $u_h^{n+1}|_K\in H^2(K)$ for every $K\in\mathcal T_h$, as is the case for the polynomial trial spaces considered below. Let
$$
\boldsymbol u_h^{n+1}
=
(u_h^{n+1},\widehat\lambda_h^{n+1})
\in\boldsymbol X_h
$$
and define the element residual
\begin{equation}\label{eq:broken-element-residual}
R_K^{n+1}
:=
f^{n+1}
-\delta_t u_h^{n+1}
+\epsilon\Delta u_h^{n+1}
-\boldsymbol\beta\cdot\nabla u_h^{n+1}
-\gamma u_h^{n+1}
\qquad\text{in }K,
\end{equation}
together with the flux residual
\begin{equation}\label{eq:broken-flux-residual}
J_K^{n+1}
:=
\widehat\lambda_h^{n+1}|_{\partial K}
-
\epsilon\nabla u_h^{n+1}\cdot n_K
\qquad\text{on }\partial K.
\end{equation}
Here $n_K$ denotes the outward unit normal to $K$, and $\widehat\lambda_h^{n+1}|_{\partial K}$ denotes the elementwise representative of the skeleton flux with the corresponding orientation. Indeed, elementwise integration by parts gives
\begin{equation}\label{eq:broken-residual-local-form}
\mathcal R_{{\rm br},h}^{n+1}
(\boldsymbol u_h^{n+1})(v)
=
\tau
\sum_{K\in\mathcal T_h}
\left[
(R_K^{n+1},v)_K
+
\langle J_K^{n+1},v\rangle_{\partial K}
\right],
\qquad
v\in V_{\rm br}.
\end{equation}
Assume, in addition, that $I_h^{\rm br}$ satisfies the local approximation estimates
$$
\|v-I_h^{\rm br}v\|_{0,K}
\lesssim
h_K\|\nabla_h v\|_{0,\omega_K},
$$
and
$$
\|v-I_h^{\rm br}v\|_{0,\partial K}
\lesssim
h_K^{1/2}\|v\|_{V_{\rm br},\omega_K},
$$
where $\omega_K$ is a uniformly bounded element patch. If the discrete flux admits an $L^2$ representative on each element boundary, the definition \eqref{eq:rho-broken-fe-def} then yields
\begin{equation}\label{eq:rho-broken-computable}
\left(\rho_{{\rm br},h}^{n+1}\right)^2
\lesssim
\tau
\sum_{K\in\mathcal T_h}
\left(
h_K^2\|R_K^{n+1}\|_{0,K}^2
+
h_K\|J_K^{n+1}\|_{0,\partial K}^2
\right).
\end{equation}

If, in addition, $I_h^{\rm br}$ preserves the local $L^2$ moments against $\mathbb P_k(K)$, namely,
$$
(q,v-I_h^{\rm br}v)_K=0,
\qquad
\forall q\in\mathbb P_k(K),\quad
\forall v\in V_{\rm br},\quad
\forall K\in\mathcal T_h,
$$
then the volume contribution can be replaced by the corresponding oscillation term. In particular,
$
h_K\|R_K^{n+1}\|_{0,K}
$
can be replaced by
$
h_K
\|(I-\pi_k^0)R_K^{n+1}\|_{0,K},
$
where $\pi_k^0$ denotes the elementwise $L^2$-projection onto
$\mathbb P_k(\mathcal T_h)$. If, moreover,
$\epsilon\nabla u_h^{n+1}\in H(\operatorname{div};\Omega)$ and
$\widehat\lambda_h^{n+1}$ is chosen as its physical diffusive normal trace, then the skeleton contribution vanishes.
\end{remark}
%
%

\section{Neural-network MinRes approximations}\label{sec:nn}

We now consider a neural-network trial class while retaining the finite-dimensional test spaces used to evaluate the residual. The analysis of Section~\ref{sec:ie-ideal} applies to any sequence of trial functions and therefore does not require the trial class to be a linear space. The only additional step is to relate the continuous residual representative entering the semidiscrete estimates to its computable discrete counterpart.

Let $\mathcal M_\Theta :=\{u_\theta:\Omega\times[0,T]\to\mathbb R:\theta\in\Theta\}$ denote the neural-network trial class, and set $u_\theta^n:=u_\theta(\cdot,t_n)$ for $n=0,\ldots,N+1$. We assume $u_\theta^n\in V$ at all time levels and define $e_\theta^n:=u^n-u_\theta^n$. We define the time-step residual functional $\mathcal R_\theta^{n+1}\in V'$ by
\begin{equation}\label{eq:nn-residual-functional}
\mathcal R_\theta^{n+1}(v)
:=
(u_\theta^n-u_\theta^{n+1},v)_0
+\tau(f^{n+1},v)_0
-\tau a(u_\theta^{n+1},v),
\qquad v\in V.
\end{equation}
Its continuous Riesz representative $\varepsilon_\theta^{n+1}\in V$ is defined by
\begin{equation}\label{eq:nn-continuous-riesz}
\tau(\varepsilon_\theta^{n+1},v)_V
=
\mathcal R_\theta^{n+1}(v),
\qquad \forall v\in V.
\end{equation}
Thus,
\begin{equation}\label{eq:nn-continuous-dual}
\sqrt{\tau}\,
\|\varepsilon_\theta^{n+1}\|_V
=
\|\mathcal R_\theta^{n+1}\|_{(\tau V)'},
\end{equation}
where $(\tau V)'$ denotes the dual norm induced by
$\sqrt{\tau}\|\cdot\|_V$.

For the computable residual decompositions below, we now specialize to the diffusion--advection--reaction realization \eqref{eq:dar-strong}. When $u_\theta^{n+1}$ has sufficient spatial regularity, we introduce the strong time-step residual
\begin{equation}\label{eq:nn-strong-residual}
r_\theta^{n+1}
:=
f^{n+1}
-\delta_tu_\theta^{n+1}
+\epsilon\Delta u_\theta^{n+1}
-\boldsymbol\beta\cdot\nabla u_\theta^{n+1}
-\gamma u_\theta^{n+1}.
\end{equation}
Whenever $\Delta u_\theta^{n+1}\in L^2(\Omega)$, the strong residual belongs to
$L^2(\Omega)$ and integration by parts gives
\begin{equation}\label{eq:nn-weak-strong-residual}
\mathcal R_\theta^{n+1}(v)
=
\tau(r_\theta^{n+1},v)_0,
\qquad \forall v\in V.
\end{equation}
We measure the error in the discrete-in-time energy quantity
\begin{equation}\label{eq:nn-energy-error}
\|e_\theta\|_{\mathcal E_\tau}^2
:=
\|e_\theta^{N+1}\|_0^2
+
\sum_{n=0}^{N}
\|e_\theta^{n+1}-e_\theta^n\|_0^2
+
\tau\alpha
\sum_{n=0}^{N}
\|e_\theta^{n+1}\|_V^2.
\end{equation}

\subsection{Bubble-enriched conforming test space}
\label{subsec:nn-bubble}

Let
$$
S_{h,0}^1
:=
\mathbb P_1(\mathcal T_h)\cap H_0^1(\Omega)
$$
and let $B_h$ denote the element-bubble space generated on each simplex $K\in\mathcal T_h$ by the standard bubble of degree $d+1$. We use the enriched conforming test space
\begin{equation}\label{eq:nn-bubble-test-space}
Y_h^{\rm bub}
:=
S_{h,0}^1+B_h
\subset V.
\end{equation}
For each time step, the discrete residual representative $\varepsilon_{\theta,h}^{n+1}\in Y_h^{\rm bub}$ is defined by
\begin{equation}\label{eq:nn-bubble-discrete-riesz}
\tau
(\varepsilon_{\theta,h}^{n+1},y_h)_V
=
\mathcal R_\theta^{n+1}(y_h),
\qquad
\forall y_h\in Y_h^{\rm bub},
\end{equation}
and we set
\begin{equation}\label{eq:nn-bubble-eta}
\eta_{\theta,h}^{n+1}
:=
\sqrt{\tau}
\|\varepsilon_{\theta,h}^{n+1}\|_V
=
\sup_{0\neq y_h\in Y_h^{\rm bub}}
\frac{
\mathcal R_\theta^{n+1}(y_h)
}{
\sqrt{\tau}\|y_h\|_V
}.
\end{equation}
To estimate the part of the residual not detected by $Y_h^{\rm bub}$, we use the moment-preserving bubble quasi-interpolant of
\cite[Lemma~6]{fuhrer2025posteriori}.
Taking $k=m=0$ in that result, there exists a linear operator $\Pi_h^{\rm bub}:V\to Y_h^{\rm bub}$ such that
\begin{align}
\|\Pi_h^{\rm bub}v\|_V
&\le C_\Pi\|v\|_V,
\label{eq:nn-bubble-proj-stability}\\
\|v-\Pi_h^{\rm bub}v\|_{0,K}
&\le C_{\rm app}h_K\|\nabla v\|_{0,\omega_K},
\label{eq:nn-bubble-proj-approx}\\
(1,v-\Pi_h^{\rm bub}v)_K
&=0,
\qquad K\in\mathcal T_h.
\label{eq:nn-bubble-proj-moment}
\end{align}
The constants depend only on the mesh's shape regularity. Let $\pi_0$ be the elementwise $L^2$-projection onto $\mathbb P_0(\mathcal T_h)$ and define
\begin{equation}\label{eq:nn-bubble-rho}
\rho_{\theta,h}^{n+1}
:=
\sqrt{\tau}\,
\bigl\|
h_{\mathcal T}(I-\pi_0)r_\theta^{n+1}
\bigr\|_0.
\end{equation}

\begin{lemma}[Discrete and continuous residuals]
\label{lem:nn-bubble-residual-comparison}
Let $\Pi_h^{\rm bub}$ be the operator described above. Then 
\begin{equation}\label{eq:nn-bubble-residual-comparison}
\eta_{\theta,h}^{n+1}
\le
\sqrt{\tau}\,
\|\varepsilon_\theta^{n+1}\|_V
\le
C_\Pi\eta_{\theta,h}^{n+1}
+
C_\rho\rho_{\theta,h}^{n+1},
\end{equation}
where $C_\rho$ depends only on the shape regularity of the mesh and the constants in \eqref{eq:nn-bubble-proj-approx}.
\end{lemma}

\begin{proof}
Since $Y_h^{\rm bub}\subset V$, restriction of the supremum in \eqref{eq:nn-continuous-dual} immediately gives
$$
\eta_{\theta,h}^{n+1}
\le
\sqrt{\tau}\|\varepsilon_\theta^{n+1}\|_V.
$$
For the converse estimate, let $v\in V$. We split
$$
\mathcal R_\theta^{n+1}(v)
=
\mathcal R_\theta^{n+1}(\Pi_h^{\rm bub}v)
+
\mathcal R_\theta^{n+1}(v-\Pi_h^{\rm bub}v).
$$
By \eqref{eq:nn-bubble-eta} and
\eqref{eq:nn-bubble-proj-stability},
$$
\big|
\mathcal R_\theta^{n+1}(\Pi_h^{\rm bub}v)
\big|
\le
C_\Pi
\eta_{\theta,h}^{n+1}
\sqrt{\tau}\|v\|_V.
$$
For the second term, using \eqref{eq:nn-weak-strong-residual} and \eqref{eq:nn-bubble-proj-moment},
$$
\mathcal R_\theta^{n+1}(v-\Pi_h^{\rm bub}v)
=
\tau
\big(
(I-\pi_0)r_\theta^{n+1},
v-\Pi_h^{\rm bub}v
\big)_0.
$$
Hence, by \eqref{eq:nn-bubble-proj-approx} and the finite overlap of the patches,
$$
\begin{aligned}
\big|
\mathcal R_\theta^{n+1}
(v-\Pi_h^{\rm bub}v)
\big|
&\le
C_\rho\tau
\big\|
h_{\mathcal T}(I-\pi_0)r_\theta^{n+1}
\big\|_0
\|v\|_V
\\
&=
C_\rho
\rho_{\theta,h}^{n+1}
\sqrt{\tau}\|v\|_V.
\end{aligned}
$$
Combining the two bounds, dividing by $\sqrt{\tau}\|v\|_V$, and taking the supremum over $v\in V\setminus\{0\}$ proves the second inequality.
\end{proof}

We now obtain the neural-network reliability estimate directly from the representation-independent stability result of Theorem~\ref{thm:standard-stability}.

\begin{theorem}[Reliability for the conforming neural-network formulation]
\label{thm:nn-bubble-reliability}
Let $\theta\in\Theta$ be such that
$
\Delta u_\theta^{n+1}\in L^2(\Omega)$, $n=0,\ldots,N$. Then
\begin{align}
\|e_\theta\|_{\mathcal E_\tau}^2
\le\;&
\|u_0-u_\theta^0\|_0^2
+
\frac{2}{\alpha}
\sum_{n=0}^{N}
\left(
C_\Pi\eta_{\theta,h}^{n+1}
+
C_\rho\rho_{\theta,h}^{n+1}
\right)^2
+
\frac{2\tau C_P^2}{\alpha}
\sum_{n=0}^{N}
\|\vartheta^{n+1}\|_0^2.
\label{eq:nn-bubble-reliability}
\end{align}
\end{theorem}

\begin{proof}
Apply Theorem~\ref{thm:standard-stability} to $u_\star^n=u_\theta^n$ and use Lemma~\ref{lem:nn-bubble-residual-comparison} at each time level.
\end{proof}

\subsection{Broken polynomial test space}
\label{subsec:nn-broken}

We next consider the broken polynomial test space
\begin{equation}\label{eq:nn-broken-test-space}
Y_{h,k}^{\rm br}
:=
\left\{
v_h\in V_{\rm br}:
v_h|_K\in\mathbb P_k(K)
\quad\forall K\in\mathcal T_h
\right\}.
\end{equation}
The index $k$ denotes the local polynomial degree; in particular, $Y_{h,0}^{\rm br}$ is the piecewise-constant test space used in the numerical experiments below. For the broken formulation, we additionally assume that
$$
\Delta u_\theta^n\in L^2(\Omega),
\qquad
n=0,\ldots,N+1,
$$
where the Laplacian is understood in the distributional sense. Since $\epsilon>0$ is constant and $u_\theta^n\in V$, this implies $\epsilon\nabla u_\theta^n\in H(\operatorname{div};\Omega)$. Hence the physical diffusive normal trace is well defined, and we set
\begin{equation}\label{eq:nn-flux-trace}
\widehat\lambda_\theta^n
:=
\operatorname{tr}_{\mathcal T_h}^{\operatorname{div}}
\bigl(\epsilon\nabla u_\theta^n\bigr)
\in\widehat\Lambda,
\qquad
\boldsymbol u_\theta^n
:=
(u_\theta^n,\widehat\lambda_\theta^n)\in U.
\end{equation}
The broken residual functional is
\begin{equation}\label{eq:nn-broken-residual-functional}
\mathcal R_{{\rm br},\theta}^{n+1}(v)
:=
(u_\theta^n-u_\theta^{n+1},v)_0
+\tau(f^{n+1},v)_0
-\tau
a_{\rm br}(\boldsymbol u_\theta^{n+1},v),
\qquad v\in V_{\rm br}.
\end{equation}
The continuous broken Riesz representative $\varepsilon_{{\rm br},\theta}^{n+1}\in V_{\rm br}$ is determined by
\begin{equation}\label{eq:nn-broken-continuous-riesz}
\tau
(\varepsilon_{{\rm br},\theta}^{n+1},v)_{V_{\rm br}}
=
\mathcal R_{{\rm br},\theta}^{n+1}(v),
\qquad
\forall v\in V_{\rm br}.
\end{equation}
Since the numerical flux in \eqref{eq:nn-flux-trace} is the physical diffusive normal trace, elementwise integration by parts yields
\begin{equation}\label{eq:nn-broken-strong-relation}
\mathcal R_{{\rm br},\theta}^{n+1}(v)
=
\tau(r_\theta^{n+1},v)_0,
\qquad
\forall v\in V_{\rm br},
\end{equation}
with $r_\theta^{n+1}$ defined by
\eqref{eq:nn-strong-residual}.

The discrete broken residual representative $\varepsilon_{{\rm br},\theta,h}^{n+1} \in Y_{h,k}^{\rm br}$ is defined by
\begin{equation}\label{eq:nn-broken-discrete-riesz}
\tau
(
\varepsilon_{{\rm br},\theta,h}^{n+1},
y_h
)_{V_{\rm br}}
=
\mathcal R_{{\rm br},\theta}^{n+1}(y_h),
\qquad
\forall y_h\in Y_{h,k}^{\rm br},
\end{equation}
and we set
\begin{equation}\label{eq:nn-broken-eta}
\eta_{{\rm br},\theta,h}^{n+1}
:=
\sqrt{\tau}
\|
\varepsilon_{{\rm br},\theta,h}^{n+1}
\|_{V_{\rm br}}.
\end{equation}
Let $\pi_k^0$ denote the elementwise $L^2$-projection onto $\mathbb P_k(\mathcal T_h)$ and define
\begin{equation}\label{eq:nn-broken-rho}
\rho_{{\rm br},\theta,h}^{n+1}
:=
\sqrt{\tau}\,
\big\|
h_{\mathcal T}
(I-\pi_k^0)
r_\theta^{n+1}
\big\|_0.
\end{equation}
For fixed $k\ge 0$, standard scaling arguments for the elementwise $L^2$-projection give
\begin{align}
\|\pi_k^0v\|_{V_{\rm br}}
&\le
C_{\Pi,{\rm br}}\|v\|_{V_{\rm br}},
\label{eq:nn-broken-proj-stab}\\
\|v-\pi_k^0v\|_{0,K}
&\le
C_{{\rm app},{\rm br}}
h_K\|\nabla v\|_{0,K},
\qquad
v\in H^1(K),
\label{eq:nn-broken-proj-approx}
\end{align}
where the constants are independent of $h$. In particular, these estimates also hold for $k=0$; in that case $\nabla_h(\pi_0^0v)=0$ on every element.
\begin{lemma}[Discrete and continuous broken residuals]
\label{lem:nn-broken-residual-comparison}
For the elementwise $L^2$-projection $\pi_k^0$ defined above,
\begin{equation}\label{eq:nn-broken-residual-comparison}
\eta_{{\rm br},\theta,h}^{n+1}
\le
\sqrt{\tau}
\|
\varepsilon_{{\rm br},\theta}^{n+1}
\|_{V_{\rm br}}
\le
C_{\Pi,{\rm br}}
\eta_{{\rm br},\theta,h}^{n+1}
+
C_{\rho,{\rm br}}
\rho_{{\rm br},\theta,h}^{n+1}.
\end{equation}
\end{lemma}

\begin{proof}
The argument is the broken analog of Lemma~\ref{lem:nn-bubble-residual-comparison}. Split the residual using the elementwise projection $\pi_k^0$, use \eqref{eq:nn-broken-proj-stab} for the component detected by the discrete test space, and use the $L^2$-orthogonality of $\pi_k^0$ together with \eqref{eq:nn-broken-proj-approx} for the complementary component.
\end{proof}
\begin{remark}[Piecewise-constant test functions]
For $k=0$, the discrete test functions are constant on each element, and therefore
$$
\|y_h\|_{V_{\rm br}}
=
\|y_h\|_0,
\qquad
y_h\in Y_{h,0}^{\rm br}.
$$
Thus $\eta_{{\rm br},\theta,h}^{n+1}$ measures the component of the strong residual detected by its elementwise averages, while $\rho_{{\rm br},\theta,h}^{n+1}$ provides a mesh-weighted measure of the complementary intra-element component. The residual comparison of Lemma~\ref{lem:nn-broken-residual-comparison} therefore remains valid for the degree-zero test space used in the numerical experiments.
\end{remark}
\begin{theorem}[Reliability for the broken neural-network formulation]
\label{thm:nn-broken-reliability}
Let $\theta\in\Theta$ satisfy
$
\Delta u_\theta^n\in L^2(\Omega)$, $n=0,\ldots,N+1$.
Then
\begin{align}
\|e_\theta\|_{\mathcal E_\tau}^2
\le\;&
\|u_0-u_\theta^0\|_0^2
+
\frac{2C_{\rm emb}^2}{\alpha}
\sum_{n=0}^{N}
\left(
C_{\Pi,{\rm br}}
\eta_{{\rm br},\theta,h}^{n+1}
+
C_{\rho,{\rm br}}
\rho_{{\rm br},\theta,h}^{n+1}
\right)^2
+
\frac{2\tau C_P^2}{\alpha}
\sum_{n=0}^{N}
\|\vartheta^{n+1}\|_0^2.
\label{eq:nn-broken-reliability}
\end{align}
\end{theorem}

\begin{proof}
Apply Theorem~\ref{thm:broken-stability} to $\boldsymbol u_\star^n=\boldsymbol u_\theta^n$ and use Lemma~\ref{lem:nn-broken-residual-comparison}.
\end{proof}
\begin{remark}[Training loss and augmented error estimator]
\label{rem:nn-training-loss}
For the broken formulation, it is useful to distinguish the quantity used as the discrete MinRes training loss from the augmented estimator used for error control. We define
\begin{equation}\label{eq:nn-training-loss}
J_{\eta}(\theta)
:=
\|u_0-u_\theta^0\|_0^2
+
\sum_{n=0}^{N}
\left(
\eta_{{\rm br},\theta,h}^{n+1}
\right)^2,
\end{equation}
and
\begin{equation}\label{eq:nn-augmented-estimator}
J_{\rm aug}(\theta)
:=
\|u_0-u_\theta^0\|_0^2
+
\sum_{n=0}^{N}
\left[
\left(
\eta_{{\rm br},\theta,h}^{n+1}
\right)^2
+
\left(
\rho_{{\rm br},\theta,h}^{n+1}
\right)^2
\right].
\end{equation}

The quantity $J_\eta$ is the accumulated discrete minimum-residual objective, whereas $J_{\rm aug}$ additionally accounts for the component of the continuous residual that is not detected by the finite-dimensional test space. Indeed, using the inequality
$
(a+b)^2\le 2a^2+2b^2
$
in Theorem~\ref{thm:nn-broken-reliability}, there exists a constant $C_{\rm rel}>0$, independent of $h$ and $\tau$, such that
\begin{equation}\label{eq:nn-augmented-reliability}
\|e_\theta\|_{\mathcal E_\tau}^2
\le
C_{\rm rel}J_{\rm aug}(\theta)
+
\frac{2\tau C_P^2}{\alpha}
\sum_{n=0}^{N}
\|\vartheta^{n+1}\|_0^2.
\end{equation}
Thus, $J_{\rm aug}$ provides a reliable error estimator up to the time-discretization defect. In contrast, the analysis does not in general provide the same reliability statement for $J_\eta$ alone.

The estimate holds for every admissible $\theta\in\Theta$ and does not require existence or uniqueness of a global minimizer of the nonconvex neural-network optimization problem.
\end{remark}
%
\section{Numerical experiments}\label{sec:numerics}

We consider three numerical experiments. The first verifies the finite-element discretization for a smooth heat problem. The second examines test-space resolution for a global space--time neural approximation. The third studies residual-driven spatial refinement for a transient diffusion--advection--reaction problem.

For the first and third experiments, we use legacy FEniCS/DOLFIN~\cite{logg2010dolfin}, which supports the restricted Raviart--Thomas facet variable used by the primal DPG formulation. The neural-network experiment uses PyTorch~\cite{paszke2019pytorch} together with FEniCSx/DOLFINx~\cite{baratta2023dolfinx}; Basix~\cite{scroggs2022basix} provides the spatial quadrature rules, while PyTorch automatic differentiation evaluates the network and its spatial derivatives.

\subsection{Smooth heat equation: finite-element verification}
\label{subsec:num-heat-fe}

We first consider the heat-equation instance of \eqref{eq:dar-strong}, obtained by taking $\epsilon=1$, $\boldsymbol\beta=\boldsymbol0$, and $\gamma=0$, on $\Omega=(0,1)^2$ with final time $T=0.1$. To avoid special behavior associated with a single Laplace eigenmode, we use the two-mode manufactured solution
$$
u(t,x,y)
=
e^{-t}\sin(\pi x)\sin(\pi y)
+
\frac12 e^{-2t}\sin(2\pi x)\sin(\pi y).
$$
The forcing term and initial condition are chosen consistently from this exact solution, which satisfies the homogeneous Dirichlet boundary condition.

For the broken-test MinRes discretization, the field variable is approximated by continuous piecewise-linear functions, the residual representative by discontinuous piecewise quadratics, and the skeleton variable by the normal-trace space induced by the lowest-order Raviart--Thomas space. In two dimensions, these polynomial degrees correspond to the $p=0$ member of the standard primal DPG family in~\cite[Section~2.2.1]{fuhrerHeuerKarkulik2021}, for which a piecewise-linear field and lowest-order skeleton trace are paired with piecewise-quadratic broken tests; the test inner product used here is, however, the steady broken $H^1$ inner product introduced above rather than the time-step-dependent norm of that reference. We therefore compare the field component directly with the classical Backward Euler finite element approximation obtained with the same conforming piecewise-linear space. The discrete initial condition is the $L^2$-projection of $u_0$ onto this field space.

We partition the unit square using an $n_x\times n_x$ structured grid, with each square subdivided into triangles, for $n_x\in\{4,8,12,16,24,32\}$. Since $h\simeq n_x^{-1}$, we choose $\tau_{\rm target}=(20n_x^2)^{-1}$, corresponding to $\tau=O(h^2)$, and set $N_T=\lceil T/\tau_{\rm target}\rceil$ and $\tau=T/N_T$ so that the final time is reached exactly. This scaling prevents the first-order Backward Euler consistency error from dominating the second-order spatial $L^2$-error. Quadrature of degree $10$ is used for the reported error and residual quantities.

For either approximation, we measure the discrete-in-time field error by
\begin{equation}\label{eq:num-full-error}
\begin{aligned}
E_h^2
:=\;&
\|e^{N_T}\|_0^2
+
\sum_{n=0}^{N_T-1}\|e^{n+1}-e^n\|_0^2
+
\tau\sum_{n=0}^{N_T-1}\|e^{n+1}\|_V^2.
\end{aligned}
\end{equation}
For the broken-test MinRes approximation, we additionally monitor the accumulated discrete residual representative
$$
\eta_h^2
:=
\|u_0-u_h^0\|_0^2
+
\tau\sum_{n=0}^{N_T-1}
\|\varepsilon_{{\rm br},h}^{n+1}\|_{V_{\rm br}}^2.
$$
In this numerical example, we use $\eta_h$ as a diagnostic of the discrete MinRes realization and do not include the complementary residual. We examine the role of the complementary residual when the discrete test space does not sufficiently resolve the continuous residual in the neural-network and adaptive experiments below.

Figure~\ref{fig:num-heat-fe} shows the convergence results. The full discrete-in-time errors of both the broken-test MinRes and standard finite element approximations converge with first order, as does $\eta_h$. The final-time $L^2(\Omega)$ errors converge with second order. On the last refinement, from $n_x=24$ to $n_x=32$, the observed rates are approximately $1.01$, $1.00$, and $1.01$ for the MinRes full error, the finite element full error, and $\eta_h$, respectively, while the corresponding final-time $L^2$ rates are approximately $2.00$ and $1.99$.

The residual representative also tracks the MinRes error closely: the ratio $\eta_h/E_h$ increases only from approximately $1.14$ on the coarsest mesh to approximately $1.20$ on the finest mesh. Thus, for this smooth finite-dimensional realization, the discrete residual representative provides a sharp numerical indicator of the full discrete-in-time error.

\begin{figure}[htbp]
  \centering
  \begin{subfigure}[t]{0.48\textwidth}
    \centering
    \includegraphics[width=\linewidth]
    {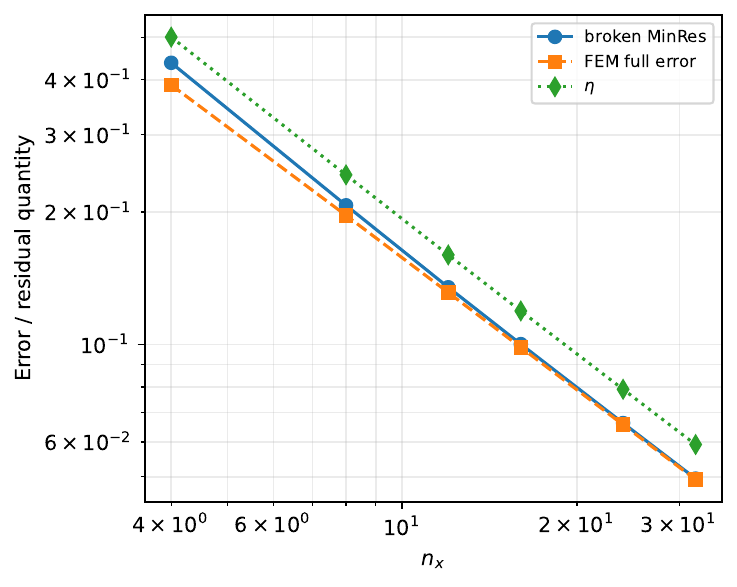}
    \caption{Full discrete-in-time error and residual.}
  \end{subfigure}
  \hfill
  \begin{subfigure}[t]{0.48\textwidth}
    \centering
    \includegraphics[width=\linewidth]
    {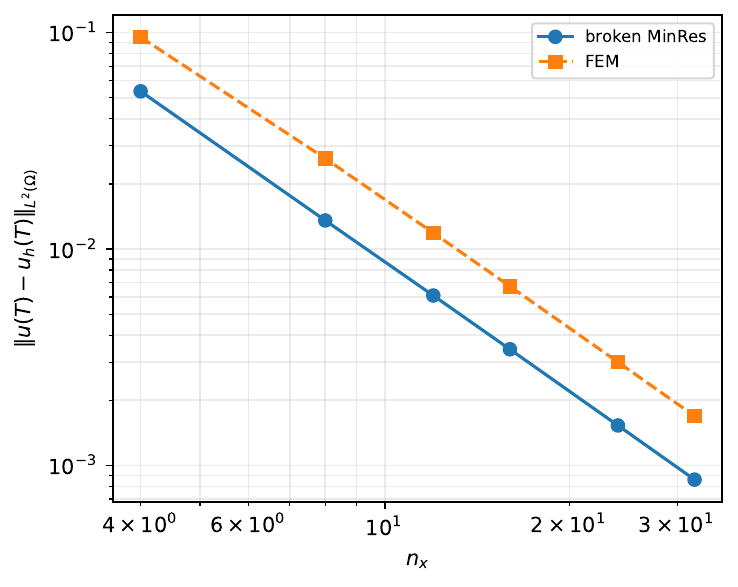}
    \caption{Final-time $L^2(\Omega)$ error.}
  \end{subfigure}
  \caption{Smooth heat equation with
  $\tau\simeq(20n_x^2)^{-1}$.}
  \label{fig:num-heat-fe}
\end{figure}

\subsection{Smooth heat equation: global space--time neural approximation}
\label{subsec:num-heat-nn}

We next consider the same manufactured heat problem as in Section~\ref{subsec:num-heat-fe}, but replace the finite-dimensional field space by a global space--time neural-network trial class. The aim of this experiment is to examine how a finite-dimensional test space resolves the residual of a neural approximation. We first train a single network using the augmented objective and then freeze both the network parameters and the time partition while refining only the spatial test space. This separates the effect of test-space refinement from neural-network optimization and time refinement. The diagnostic is motivated by the complementary-residual analysis in~\cite{fuhrer2025posteriori} and by adaptive test-space enrichment for neural approximations in~\cite{udomworarat2026neural}, but here the trained approximation is kept fixed during the refinement study.

A single network $u_\theta:\Omega\times[0,T]\to\mathbb R$ represents the approximation at all time levels. We use a fully connected network with five hidden layers of width $40$, hyperbolic-tangent activation, and a linear output layer. Homogeneous Dirichlet conditions are imposed strongly by setting
$$
u_\theta(x,y,t)
=
x(1-x)y(1-y)\,N_\theta(x,y,t).
$$
The resulting network contains $6761$ trainable parameters. All computations are performed in double precision. The smooth $\tanh$ activation is consistent with the regularity required in Section~\ref{sec:nn}, since the strong residual contains $\Delta u_\theta$, which is evaluated by automatic differentiation.

We use the broken piecewise-constant test space $Y_{h,0}^{\rm br}$. For this choice, the discrete Riesz representative is determined by the cell averages of the strong time-step residual, whereas the complementary contribution measures the mesh-weighted part of the residual not represented by the discrete test space; see Section~\ref{subsec:nn-broken}. We train with the augmented objective $J_{\rm aug}$ introduced in Remark~\ref{rem:nn-training-loss}, rather than with $J_\eta$ alone.

To obtain the neural approximation used below, we use a coarse-to-fine continuation over $n_x\in\{4,6,8\}$. At each stage we set $\tau_{\rm target}=(20n_x^2)^{-1}$, $N_T=\lceil T/\tau_{\rm target}\rceil$, and $\tau=T/N_T$. Although $u_\theta$ is a global space--time approximation, the loss is assembled only from spatial quadrature at the discrete Backward Euler time levels. Consequently, no quadrature over the full space--time cylinder $\Omega\times(0,T)$ is required.

We initially evaluate the spatial integrals in the training objective with a degree-$7$ quadrature rule. A richer degree-$10$ rule is used to monitor the same quantities; we append the superscript ``ref'' to quantities evaluated with this richer rule. This additional quadrature check is motivated by the sensitivity of variational neural methods to numerical integration; see, e.g.,~\cite{berrone2022quadratures}. If the relative discrepancy between the training and reference evaluations of $J_{\rm aug}$ exceeds $5\%$, the degree-$10$ rule is promoted to the training rule and a degree-$14$ rule is used as the new reference rule. This safeguard is activated on the coarsest test mesh. For the optimization-history diagnostics below, we write
$$
(\rho^{\rm ref})^2
:=
\sum_{n=0}^{N_T-1}
\left(
\rho_{{\rm br},\theta,h}^{n+1,{\rm ref}}
\right)^2,
$$
and use $J_\eta^{\rm ref}$ and $J_{\rm aug}^{\rm ref}$ for the corresponding reference-quadrature evaluations of the quantities in Remark~\ref{rem:nn-training-loss}.

On the $n_x=4$ mesh, the network is initialized with a fixed Xavier-uniform initialization and trained for $2500$ full-batch Adam iterations with learning rate $10^{-3}$. We then switch to full-batch L-BFGS with a strong-Wolfe line search. The quasi-Newton stage is applied in blocks of $20$ iterations and continued until the relative reduction of $J_{\rm aug}^{\rm ref}$ is below $5\times10^{-3}$ for three successive blocks. On the finer training meshes, the optimized parameters from one stage are used to initialize the next stage, followed by L-BFGS with the same stopping criterion. No interpolation of the neural trial function is required because the trial representation is independent of the discrete test space.

The final continuation stage has $n_x=8$, $N_T=128$, and $\tau=T/128$. From this point onward, no further optimization is performed. We freeze the trained network and this time partition and evaluate the same time-discrete strong residual on the uniformly refined sequence of piecewise-constant test spaces associated with
$$
n_x=4,\;8,\;16,\;32.
$$
All quantities in this post-processing study are evaluated using degree-$14$ quadrature. We retain the notation $u_\theta$ for this fixed trained network.

Since only the broken formulation is considered in the remainder of this experiment, we suppress the subscript ${\rm br}$ in the accumulated quantities. We also separate the residual contribution from the initial-condition contribution contained in $J_\eta$ and $J_{\rm aug}$. Define the accumulated strong-residual norm
$$
\mathsf R_{\theta,\tau}^2
:=
\tau\sum_{n=0}^{N_T-1}
\|r_\theta^{n+1}\|_0^2,
$$
and the accumulated component detected by the piecewise-constant test space
$$
\eta_{\theta,h}^2
:=
\sum_{n=0}^{N_T-1}
\left(
\eta_{{\rm br},\theta,h}^{n+1}
\right)^2
=
\tau\sum_{n=0}^{N_T-1}
\|\pi_0^0r_\theta^{n+1}\|_0^2.
$$
For this experiment, we additionally introduce the unweighted unresolved strong-residual component
$$
\zeta_{\theta,h}^2
:=
\tau\sum_{n=0}^{N_T-1}
\|(I-\pi_0^0)r_\theta^{n+1}\|_0^2.
$$
Thus, $\zeta_{\theta,h}$ measures the intra-element part of the strong residual that is not represented by the piecewise-constant test space. By the $L^2$-orthogonality of $\pi_0^0$,
$$
\mathsf R_{\theta,\tau}^2
=
\eta_{\theta,h}^2+\zeta_{\theta,h}^2.
$$
In contrast, the accumulated complementary contribution appearing in the reliability estimate is
$$
\rho_{\theta,h}^2
:=
\sum_{n=0}^{N_T-1}
\left(
\rho_{{\rm br},\theta,h}^{n+1}
\right)^2
=
\tau\sum_{n=0}^{N_T-1}
\left\|
h_{\mathcal T}(I-\pi_0^0)r_\theta^{n+1}
\right\|_0^2.
$$
Hence, $\zeta_{\theta,h}$ is an unweighted diagnostic of unresolved strong-residual content, whereas $\rho_{\theta,h}$ is the mesh-weighted complementary contribution entering the reliability estimate.

Figure~\ref{fig:num-heat-nn-residual} shows the test-space-refinement study. Since the neural approximation and time partition are fixed, $\mathsf R_{\theta,\tau}$ is independent of the test space up to quadrature error. Numerically, it remains essentially constant at $4.10\times10^{-3}$, with relative variation below $5\times10^{-5}$ over the four meshes. The resolved fraction $\eta_{\theta,h}^2/\mathsf R_{\theta,\tau}^2$ is approximately $0.12$, $0.43$, $0.73$, and $0.89$ for $n_x=4,8,16,32$, respectively, while the unresolved fraction $\zeta_{\theta,h}^2/\mathsf R_{\theta,\tau}^2$ decreases correspondingly from approximately $0.88$ to $0.11$. Thus, refinement of the test space detects an increasingly large portion of the same fixed strong residual.

The weighted complementary contribution decreases even more rapidly: $\rho_{\theta,h}/\eta_{\theta,h}$ is approximately $0.96$, $0.21$, $0.054$, and $0.015$ for $n_x=4,8,16,32$, respectively. Since $\rho_{\theta,h}$ contains the factor $h_{\mathcal T}$, its decrease alone would not demonstrate improved residual resolution. The simultaneous decrease of the unweighted unresolved fraction shows that test-space refinement reduces the portion of the fixed residual that is invisible to the piecewise-constant test space.

\begin{figure}[htbp]
  \centering
  \begin{subfigure}[t]{0.48\textwidth}
    \centering
    \includegraphics[width=\linewidth]
    {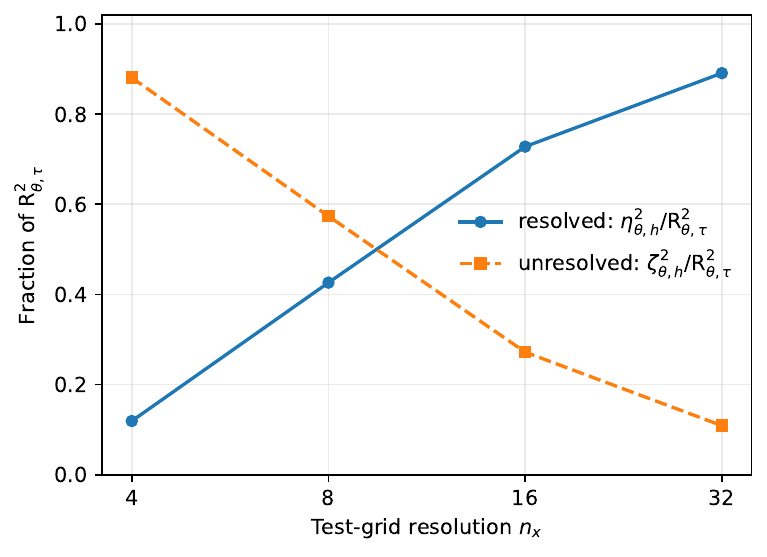}
    \caption{Resolved and unresolved fractions of the accumulated strong residual.}
  \end{subfigure}
  \hfill
  \begin{subfigure}[t]{0.48\textwidth}
    \centering
    \includegraphics[width=\linewidth]
    {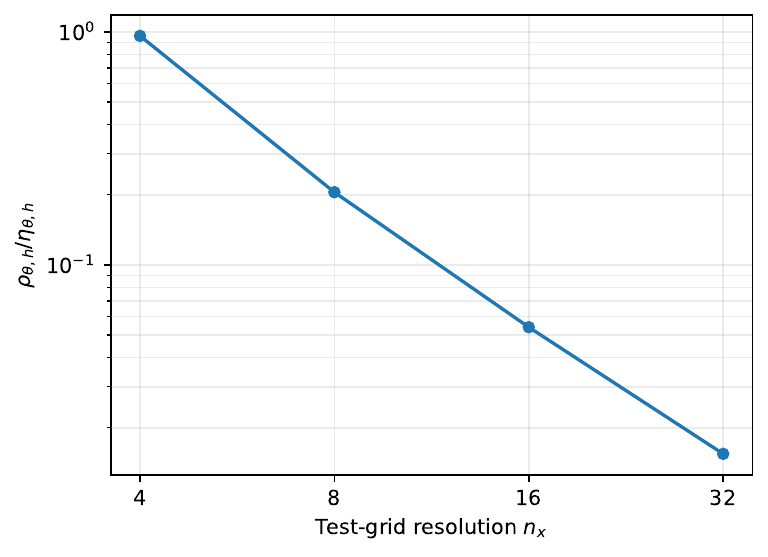}
    \caption{Weighted complementary contribution relative to the detected residual.}
  \end{subfigure}
  \caption{Test-space refinement for a fixed global space--time neural approximation. After training, the network parameters and the time partition ($N_T=128$) are frozen, and the same time-discrete strong residual is evaluated on the uniformly refined sequence of piecewise-constant test spaces associated with $n_x=4,8,16,32$. Panel (a) shows the resolved and unresolved fractions of the accumulated strong-residual squared norm. Panel (b) shows the ratio of the weighted complementary contribution $\rho_{\theta,h}$ to the accumulated detected residual contribution $\eta_{\theta,h}$. We use degree-$14$ quadrature throughout this post-processing study.}
  \label{fig:num-heat-nn-residual}
\end{figure}

Figure~\ref{fig:num-heat-nn-training} shows the optimization history on the initial $n_x=4$ training stage. Adam first reduces the augmented objective, and the subsequent L-BFGS stage further substantially decreases the residual quantities and the true error. During the optimization, $\rho^{\rm ref}/\sqrt{J_\eta^{\rm ref}}$ increases and eventually exceeds one. On this coarse test mesh, the complementary contribution therefore becomes comparable with the unaugmented quantity $\sqrt{J_\eta^{\rm ref}}$ even after the augmented objective has been substantially reduced.

\begin{figure}[htbp]
  \centering
  \begin{subfigure}[t]{0.48\textwidth}
    \centering
    \includegraphics[width=\linewidth]
    {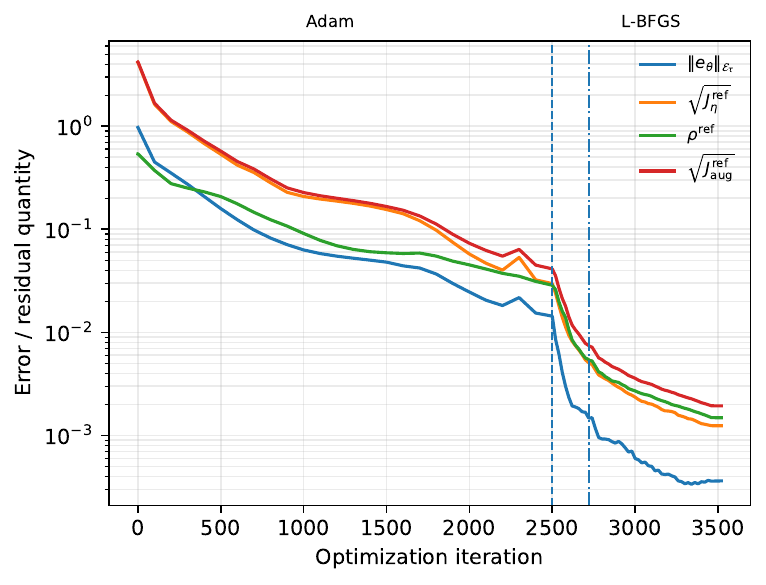}
    \caption{Optimization history.}
  \end{subfigure}
  \hfill
  \begin{subfigure}[t]{0.48\textwidth}
    \centering
    \includegraphics[width=\linewidth]
    {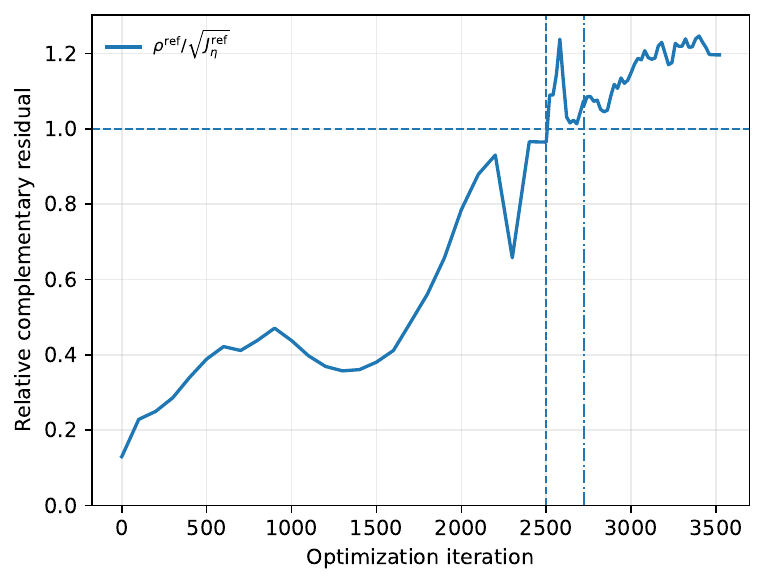}
    \caption{Relative complementary residual.}
  \end{subfigure}
  \caption{Optimization on the $n_x=4$ test mesh.
  Dashed/dash-dotted vertical lines mark the Adam--L-BFGS transition and the quadrature update, respectively. In the right panel, the horizontal line corresponds to $\rho^{\rm ref}=\sqrt{J_\eta^{\rm ref}}$.}
  \label{fig:num-heat-nn-training}
\end{figure}

\FloatBarrier

\subsection{Diffusion--advection--reaction problem:
estimator-driven spatial marking}
\label{subsec:num-dar}

The final experiment examines whether local residual contributions can drive spatial refinement for a transient nonsymmetric problem. We use the diffusion--advection--reaction realization \eqref{eq:dar-strong} on $\Omega=(0,1)^2$ with $\epsilon=10^{-2}$, $\boldsymbol\beta=(1,0)^\top$, $\gamma=1$, and $T=0.2$. Since $\nabla\cdot\boldsymbol\beta=0$ and $\gamma>0$, the coercivity condition \eqref{eq:dar-coercivity-assumption} is satisfied. We emphasize that this experiment concerns these fixed coefficients; we do not claim parameter robustness as $\epsilon\to0$.

To obtain a problem for which spatial localization changes throughout the time interval, we use a variant of the moving-peak manufactured solutions employed in adaptive computations for parabolic problems; see, e.g.,~\cite{stevenson2022wavelet}. Specifically, we take
$$
u(t,x,y)
=
x(1-x)y(1-y)
\exp\!\left(
-\kappa\left[
(x-x_c(t))^2+\left(y-\frac12\right)^2
\right]
\right),
$$
where $x_c(t)=0.2+0.6\,t/T$ and $\kappa=80$. Thus, the localized feature travels horizontally from $(0.2,0.5)$ to $(0.8,0.5)$. The forcing term in \eqref{eq:dar-strong} is obtained by substituting the manufactured solution into the differential equation.

The field variable is approximated by continuous piecewise-linear functions and the broken residual representative by discontinuous piecewise quadratics. The skeleton-flux unknown is represented by the normal trace of a restricted lowest-order Raviart--Thomas field $\boldsymbol\sigma_h^n$, so that, elementwise,
$$
\widehat\lambda_h^n|_{\partial K}
=
\boldsymbol\sigma_h^n\cdot n_K.
$$
We project the initial condition onto the conforming field space. All reported residual and error quantities are evaluated using quadrature of degree $10$.

We use a fixed Backward Euler partition with $N_T=200$ and $\tau=10^{-3}$ on every spatial mesh. The refinement is therefore purely spatial. More precisely, each adaptive cycle uses a \emph{single spatial mesh} for the complete interval $[0,T]$. We first solve the full transient problem on that mesh, accumulate local residual contributions over all time steps, and then mark and refine cells. We then recompute the full time-dependent problem from $t=0$ on the refined mesh. Thus we do not use a different spatial mesh on each time slab, and no inter-mesh transfer of intermediate-time solutions is required.

For an element $K\in\mathcal T_h$, we accumulate the discrete residual contribution
$$
\eta_K^2
:=
\|u_0-u_h^0\|_{0,K}^2
+
\tau\sum_{n=0}^{N_T-1}
\|\varepsilon_{{\rm br},h}^{n+1}\|_{V_{\rm br}(K)}^2.
$$
The complementary contribution is constructed from the element strong residual
$$
R_K^{n+1}
=
f^{n+1}
-\frac{u_h^{n+1}-u_h^n}{\tau}
+\epsilon\Delta u_h^{n+1}
-\boldsymbol\beta\cdot\nabla u_h^{n+1}
-\gamma u_h^{n+1}
$$
and the physical-flux mismatch
$$
J_{K,F}^{n+1}
=
\bigl(
\boldsymbol\sigma_h^{n+1}
-\epsilon\nabla u_h^{n+1}
\bigr)\cdot n_K
\qquad
\text{on }F\subset\partial K.
$$
Motivated by the localized upper bound in Remark~\ref{rem:rho-broken-localization}, we use the computable local surrogate
$$
\rho_K^2
:=
\tau\sum_{n=0}^{N_T-1}
\left[
h_K^2\|R_K^{n+1}\|_{0,K}^2
+
\sum_{F\subset\partial K}
h_F\|J_{K,F}^{n+1}\|_{0,F}^2
\right].
$$
We do not estimate the mesh-independent constants implicit in the theoretical localization bound; instead, we use the canonical element and facet scalings to construct the marking distribution. For interior facets, we assign the two one-sided mismatch contributions to their corresponding neighboring elements. We then define $\xi_K^2:=\eta_K^2+\rho_K^2$.

We compare three refinement strategies: uniform refinement, D\"orfler marking~\cite{doerfler1996convergent} based on $\eta_K^2$, and D\"orfler marking based on $\xi_K^2$. For the two residual-based strategies, the marked set $\mathcal M_h$ is chosen so that
$$
\sum_{K\in\mathcal M_h} I_K^2
\ge
\theta
\sum_{K\in\mathcal T_h} I_K^2,
\qquad
\theta=0.5,
$$
where $I_K^2=\eta_K^2$ or $I_K^2=\xi_K^2$, respectively. All strategies start from a $4\times4$ subdivision of the unit square, corresponding to $32$ triangles. We perform six successive refinements for each residual-based strategy. We stop uniform refinement at $2048$ cells, since the next uniform level would exceed the prescribed limit of $5000$ cells.

For the present coefficients, since
$\nabla\cdot\boldsymbol\beta=0$,
$$
a_{\rm dar}(v,v)
=
\epsilon\|\nabla v\|_0^2
+
\gamma\|v\|_0^2.
$$
Accordingly, retaining the coercive bilinear form contribution in the energy argument of Theorem~\ref{thm:broken-stability}, rather than replacing it only by its lower bound $\alpha\|v\|_V^2$, motivates the following discrete-in-time error quantity:
$$
\begin{aligned}
E_{\rm dar}^2
:=\;&
\|e^{N_T}\|_0^2
+
\sum_{n=0}^{N_T-1}
\|e^{n+1}-e^n\|_0^2
+
\tau\sum_{n=0}^{N_T-1}
\left(
\epsilon\|\nabla e^{n+1}\|_0^2
+
\gamma\|e^{n+1}\|_0^2
\right).
\end{aligned}
$$
We also construct a local reference error distribution solely for diagnostic purposes. Since $\eta_K^2$ contains the initial projection contribution, for this localization diagnostic we define
$$
E_{\rm ref}^2
:=
\|u_0-u_h^0\|_0^2+E_{\rm dar}^2,
$$
and let $\mathcal E_K^2$ denote the contribution of element $K$ to $E_{\rm ref}^2$, including the initial-projection, time-increment, energy, and final-time terms. The adaptive algorithm never uses this quantity. Global effectivity alone does not measure whether an estimator localizes the error in the correct elements; related localization diagnostics have therefore been used in adaptive finite-element studies, for example the fraction of incorrect refinement decisions in~\cite{izsak2008implicit}. To assess how well a marking indicator reproduces the spatial distribution of the reference error, we use the scale-invariant cosine alignment
$$
\mathcal A(I,\mathcal E)
:=
\frac{
\sum_{K\in\mathcal T_h} I_K^2\mathcal E_K^2
}{
\left(\sum_{K\in\mathcal T_h} I_K^4\right)^{1/2}
\left(\sum_{K\in\mathcal T_h} \mathcal E_K^4\right)^{1/2}
}.
$$
Thus $\mathcal A(I,\mathcal E)$ is the cosine similarity between the nonnegative vectors $(I_K^2)_K$ and $(\mathcal E_K^2)_K$. In particular, $0\le\mathcal A(I,\mathcal E)\le1$, with values close to one indicating similar spatial distributions. This diagnostic compares only the spatial distributions and is insensitive to their overall magnitudes.

\begin{figure}[!t]
  \centering
  \begin{subfigure}[t]{0.48\textwidth}
    \centering
    \includegraphics[width=\linewidth]
    {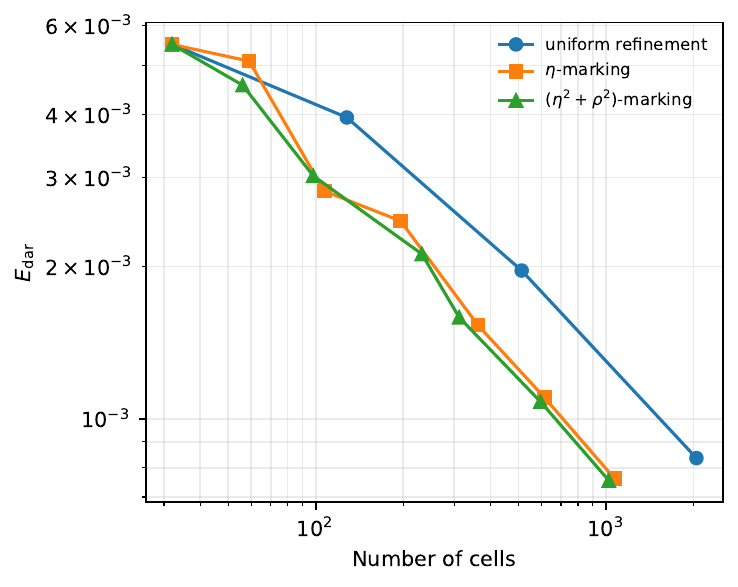}
    \caption{Error versus spatial complexity.}
  \end{subfigure}
  \hfill
  \begin{subfigure}[t]{0.48\textwidth}
    \centering
    \includegraphics[width=\linewidth]
    {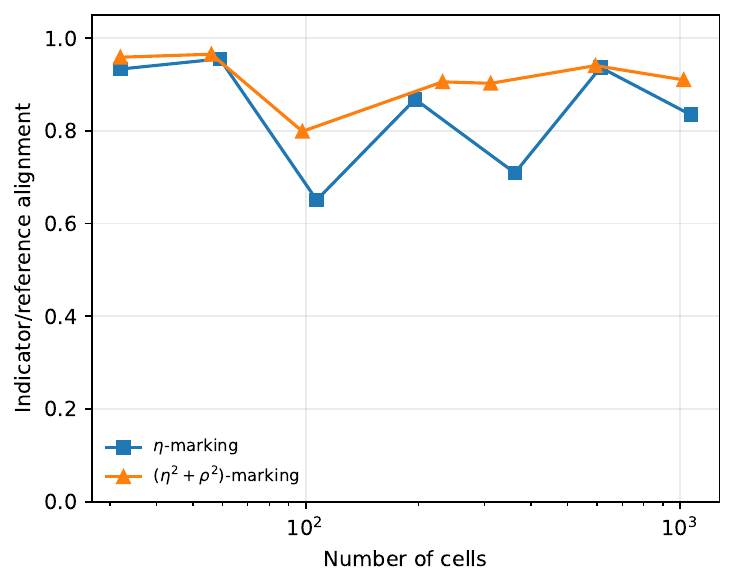}
    \caption{Indicator/reference alignment.}
  \end{subfigure}
  \caption{Adaptive diffusion--advection--reaction experiment.}
  \label{fig:num-dar-performance}
\end{figure}

Figure~\ref{fig:num-dar-performance} summarizes the quantitative results. Both residual-driven strategies improve the error-versus-complexity behavior relative to uniform refinement. On the finest meshes, $\eta$-marking gives $E_{\rm dar}=7.62\times10^{-4}$ with $1068$ cells, while augmented marking gives $E_{\rm dar}=7.55\times10^{-4}$ with $1022$ cells. For comparison, uniform refinement gives $E_{\rm dar}=8.36\times10^{-4}$ with $2048$ cells. Thus, both adaptive procedures achieve slightly lower errors using about half as many cells as the finest uniform mesh.

The difference between the two adaptive indicators is clearer in their spatial localization. Along the computed adaptive sequences, the augmented indicator remains consistently more closely aligned with the manufactured local error than the $\eta$-based indicator. For example, on the final meshes, $\mathcal A(\eta,\mathcal E)\approx0.835$ and $\mathcal A(\xi,\mathcal E)\approx0.910$. The improvement is more pronounced on some intermediate meshes, where the distribution based on $\eta_K^2$ temporarily loses alignment with the reference local error, whereas the augmented indicator remains comparatively stable.

The final adaptive meshes are shown in Figure~\ref{fig:num-dar-meshes}. Both strategies concentrate refinement along the trajectory of the moving localized feature. The resulting meshes have comparable overall complexity.

\begin{figure}[!t]
  \centering
  \begin{subfigure}[t]{0.42\textwidth}
    \centering
    \includegraphics[width=\linewidth]
    {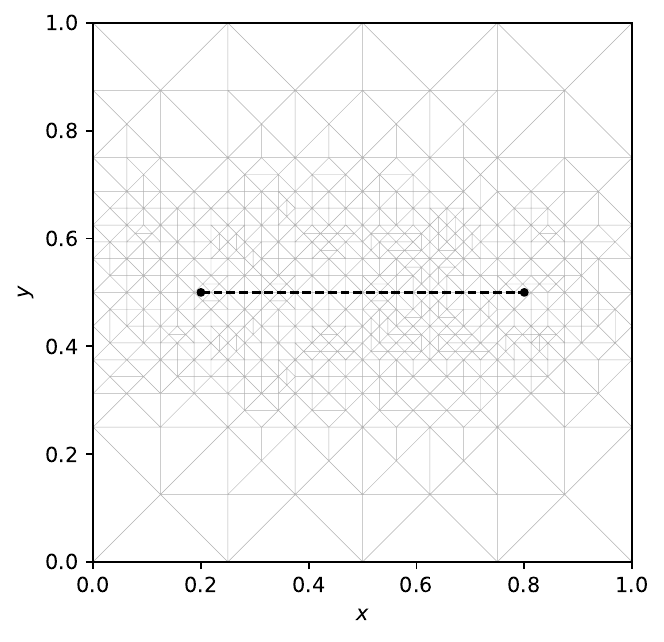}
    \caption{$\eta$-marking.}
  \end{subfigure}
  \hfill
  \begin{subfigure}[t]{0.42\textwidth}
    \centering
    \includegraphics[width=\linewidth]
    {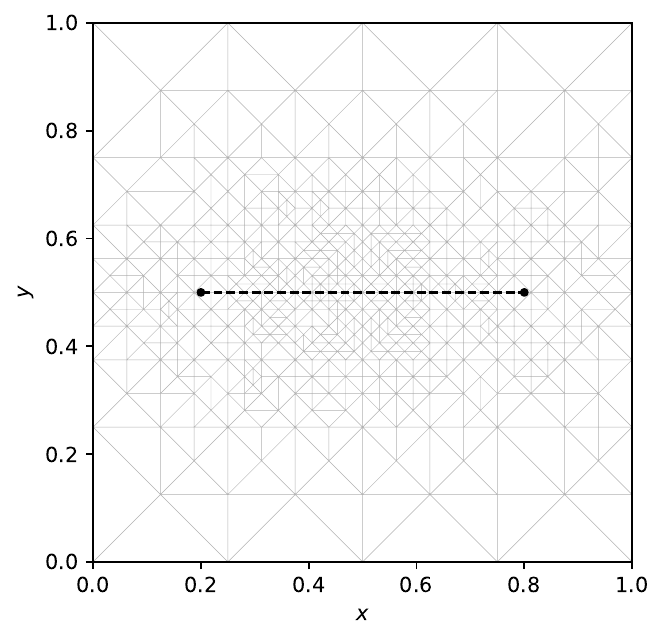}
    \caption{$(\eta^2+\rho^2)$-marking.}
  \end{subfigure}
  \caption{Final adaptive meshes. The dashed line indicates the
  trajectory of the moving feature.}
  \label{fig:num-dar-meshes}
\end{figure}

\FloatBarrier
\section{Conclusions}

We proposed a time-stepping minimum-residual framework for coercive transient variational problems in which the residual at each Backward Euler step is represented with the steady spatial test-space inner product, scaled by the time step. For arbitrary trial sequences, this yields stability estimates in a discrete parabolic energy quantity and one-sided efficiency estimates. After discretization of the test space, the computable residual representative is supplemented by a complementary residual contribution, yielding fully discrete reliability estimates for both conforming and broken-test formulations. The same framework applies to neural-network trial classes, for which the residual decompositions considered here are computable from spatial quadratures at the discrete time levels.

For the smooth heat problem, the broken-test MinRes discretization exhibits the expected convergence rates and the accumulated discrete residual tracks the discrete-in-time error. For the neural approximation, test-space refinement resolves an increasing fraction of a fixed residual, while the relative complementary contribution decreases. In the transient diffusion--advection--reaction example, residual-based marking improves error versus spatial complexity compared with uniform refinement, and the augmented indicator provides more consistent localization of the reference error along the computed adaptive sequences.

The efficiency estimates obtained here are not uniform reverse bounds for the stability quantity as $\tau\to0$, and in the broken setting they involve the full trial-space error, including the skeleton component. The adaptive computations also use one spatial mesh over the whole time interval at each refinement cycle. Extensions of interest include time-dependent spatial meshes, parameter-robust formulations for convection-dominated problems, nonlinear problems in Banach spaces using duality maps~\cite{muga2026residual}, and nonconforming or hybrid spatial discretizations.

\bibliographystyle{abbrv}
\bibliography{references}
\end{document}